\documentclass[a4paper,11pt]{amsart}

\usepackage{verbatim}

\usepackage[T1]{fontenc}
\usepackage[latin1]{inputenc}
\usepackage{setspace}
\usepackage{indentfirst}

\usepackage{geometry}
\usepackage[hyperfootnotes=false]{hyperref}

\usepackage{times}
\usepackage{amsthm}
\usepackage{amsmath}
\usepackage{amsfonts}
\usepackage{amssymb}
\usepackage{bibentry}
\usepackage{color}
\usepackage{mathrsfs}
\usepackage{stmaryrd}
\SetSymbolFont{stmry}{bold}{U}{stmry}{m}{n}

\usepackage{bm}

\usepackage[all]{xy}
\usepackage{graphicx}
\usepackage{subfigure}
\usepackage{todonotes}
\usepackage[toc,page]{appendix}

\newcommand{\tl}{\triangleleft}

\DeclareMathOperator{\id}{id}
\DeclareMathOperator{\Hom}{Hom}

\DeclareMathOperator{\im}{Im}

\DeclareMathOperator{\Diff}{Diff}
\DeclareMathOperator{\GL}{GL}

\DeclareMathOperator{\Aut}{Aut}

\DeclareMathOperator{\pt}{pt}
\DeclareMathOperator{\GA}{GA}

\DeclareMathOperator{\Span}{Span}

\DeclareMathOperator{\Lie}{Lie}
\DeclareMathOperator{\pr}{pr}

\newtheorem{Thm}{Theorem}[section]

\newtheorem{Pro}[Thm]{Proposition}
\newtheorem{Lem}[Thm]{Lemma}
\newtheorem{Cor}[Thm]{Corollary}
\newtheorem{Def-Pro}[Thm]{Definition-Proposition}
\newtheorem{Def}[Thm]{Definition}

\theoremstyle{definition}
\newtheorem{Ex}[Thm]{Example}
\newtheorem{Rm}[Thm]{Remark}

\begin{document}

\title{Double Principal Bundles}

\author{Honglei Lang}
\address{Department of Applied Mathematics, China Agricultural University, Beijing, China}
\email{hllang@cau.edu.cn}
\author{Yanpeng Li}
\address{Department of Mathematics, Sichuan University, Chengdu, China}
\email{yanpeng.zero@gmail.com}
\author{Zhangju Liu}
\address{Department of Mathematics, Henan University, Henan, China}
\email{liuzj@pku.edu.cn}


\begin{abstract}
 	We define double principal  bundles (DPBs), for which the frame bundle of a double vector bundle, double Lie groups and double homogeneous spaces are basic examples. It is shown that a double vector bundle can be realized as the associated bundle of its frame bundle. Also dual structures, gauge transformations and connections in DPBs are investigated.
\end{abstract}

\maketitle

\hfill
To the memory of Kirill C. H. Mackenzie
\vspace{2.5em}

\section{Introduction}

Double structures arise naturally in Poisson geometry, where the symplectic double groupoid was constructed by Lu and Weinstein \cite{lu} as the double of a Poisson Lie group. As basic objects in second-order geometry, general  structures of double Lie algebroids and double Lie groupoids were studied by Mackenzie in \cite{Mackenzie1992Double, Mackenzie, Mackenzie2000Double}, where double Lie algebroids can also be used to describe  the double of a Lie bialgebroid. See \cite{BCH,gracia2010lie,mackenzie2005duality, voronov2012structure} for more details on double Lie theory.
   
As a fundamental double structure, which genuinely appears in the theory of classical mechanics \cite{MT} and also in other branches of mathematics, the double vector bundle (DVB) has been extensively studied, {\em i.e.}, \cite{CLS2014, grabowski2010double, notes, konieczna1999double, Mackenzie1999, roytenberg2002structure}. The duality theory of DVBs was analyzed in \cite{konieczna1999double} and \cite{Mackenzie1999}.  In \cite{roytenberg2002structure}, a Courant algebroid was constructed by the symplectic realization from a cotangent double vector bundle.

Principal bundles play a fundamental role in differential geometry. To any given principal bundle one can  associate a vector bundle. Meanwhile, a class of  important  Lie algebroids and Lie groupoids appear due to the symmetry of a principal bundle. It is therefore natural to look for a similar object as a principal bundle in second-order geometry, from which one can construct double vector bundles, double Lie algebroids and double Lie groupoids.

Given the motivation above, we present in this paper the notion of a double principal bundle (DPB), a new kind of double structures, defined by a commutative square of principal bundles with certain compatibility conditions. A class of DPBs is equipped with structure groups, which we call double Lie groups. The frame bundle of a DVB is such an example. Unlike usual principal bundles, a DPB does not necessarily admit a structure group. A counterexample is given in Section \ref{DLGDLA}. We show that a DVB can always be realized as an associated bundle of a DPB. With this construction, it is not hard to get the dual DVB by taking dual representation of the frame bundle of a DVB. It is hidden in the literature, see \cite{voronov2012structure} for example, that the frame bundles of DVBs should carry a structure of DPB. See Remark \ref{Rm:4.12} for details. Double Lie groups in our construction are different from the known double Lie groupoids over a point in \cite{Mackenzie1992Double, Mackenzie2000Double}. We will examine the relations of these two kinds of double structures by examples.

This paper is organized as follows: In section \ref{DPFB}, we give the definition of double principal bundles and basic examples such as double Lie groups and double homogeneous spaces. In section \ref{DLGDLA}, we expound  the structures  of double Lie groups. In particular, the automorphism group of a double vector space is given, which will be used to construct  the frame bundles of double vector bundles. The main result of this paper is formulated in Section \ref{ADVB}. Namely, a double vector bundle can be realized as an associated bundle of a double principal  bundle, with the fiber being a double vector space.  Also the dual structures  are investigated in this section. In Section \ref{App}, we study gauge transformations and connections in  double principal bundles.


\section{Double Principal Bundles}\label{DPFB}

In this section, we shall introduce the notion of double principal bundles (DPBs) and give some basic properties. Also the double homogeneous spaces will be considered, as examples of DPBs.

\begin{Def}
	A \emph{double principal bundle} $\mathbb{P}=(P;P_1,P_2;M)$ consists of  four principal bundles and two exact sequences of their structure  groups:
	\begin{equation}\label{def of DPB}
		\vcenter{\xymatrix{
					P \ar[d]^{\pi_1} \ar[r]^{\pi_2} &P_2 \ar[d]^{\lambda_2}\\
					P_1 \ar[r]^{\lambda_1} &M
				}}, \quad
		\vcenter{\xymatrix @C=2.0pc @R=.5pc{
					& & K_1 \ar[r]^{\phi_2} & G_2 \ar[rd] &\\
					1\ar[r] & G_0 \ar[ru]^{i_1} \ar[rd]^{i_2} & &  &1\\
					& & K_2\ar[r]^{\phi_1} & G_1 \ar[ru]  &
				}}.
	\end{equation}
	where $P_1(G_1,M)$ and $P_2(G_2,M)$ are principal bundles over $M$.  The total space $P$ has two principal bundle structures, denoted  by $P_{I} =P(K_1,P_1)$, $P_{II}=P(K_2,P_2)$, over $P_1$, $P_2$ respectively, such that 
	\begin{itemize}
		\item $(\pi_1,\phi_1): P(K_2,P_2)\to P_1(G_1,M)$ and $(\pi_2,\phi_2): P(K_1,P_1)\to  P_2(G_2,M)$ are homomorphisms of principal bundles;
		\item $p\cdot i_1(g_0)=p\cdot i_2 (g_0), \quad \forall p \in P, \, \, \forall g_0 \in G_0$.  
	\end{itemize}
\end{Def}

It is immediate that the fiber product $(P_1\times_M P_2;P_1,P_2;M)$ forms a DPB for any two principal bundles $P_1$ and $P_2$ over $M$. The following proposition shows that a DPB is always the extension of this fiber product. Let  $\widehat{K}_i$ be the subgroup of $\Diff(P)$ induced by the action of  $K_i (i= 1,2)$ on $P$ and let  
\[
	\widehat{G}_0 : = \widehat{i_1(G_0)}  = \widehat{i_2(G_0)} \subset \Diff(P).
\] 
For convenience, we use $pk_i$ and $pg_0$ to denote the actions of $k_i\in K_i$ and $g_0\in G_0$ on $p\in P$.
\begin{Pro}
	Given a DPB $\mathbb{P}=(P;P_1,P_2;M)$, we have that $ \widehat{G}_0=\widehat{K}_1\cap \widehat{K}_2$ and the total space $P$  is a principal $G_0$-bundle over $P_1\times_M P_2$ with the projection $(\pi_1,\pi_2)$:
	\begin{equation}\label{fig1}
		\vcenter{\xymatrix @C=2.0pc @R=.7pc{
		& & P_2 \ar[dr] &\\
		G_0 \to P\, \ar[r] \ar@/^3ex/[urr] \ar@/_3ex/[drr] & P_1\times_M P_2 \ar[dr] \ar[ur] & &M\\
		& & P_1 \ar[ur] &
		}}.
	\end{equation}
	The  principal bundle $P_C:=P(G_0,P_1\times_M P_2)$ is called the \emph{core} of $\mathbb{P}$.   
\end{Pro}

\begin{proof}
	By definition, we have $\widehat{G}_0\subset \widehat{K}_1\cap\widehat{K}_2$. For the other direction, suppose $\widehat{k}_1=\widehat{k}_2\in \widehat{K}_1\cap\widehat{K}_2$, then following from the free action of $G_1$ on $P_1$
	and
	\[
		\pi_1(p)=\pi_1(pk_1)=\pi_1(pk_2)=\pi_1(p)\phi_1(k_2),\qquad \forall p\in P,
	\]
	we get $\phi_1(k_2)=e$, i.e., $k_2\in \ker \phi_1$ and thus $\widehat{k}_2\in \widehat{G}_0$. This proves $ \widehat{G}_0=\widehat{K}_1\cap \widehat{K}_2$.
	
    Then we prove $P_C$ is a principal bundle. First note that $(\pi_1,\pi_2): P\to P_1\times_M P_2$ is a surjective submersion with the kernel being $G_0$-orbits. In fact, for any $(p_1,p_2)\in P_1\times_M P_2$, we have $(\pi_1,\pi_2)^{-1}(p_1,p_2)=F_{p_1}\cap F_{p_2}$, the intersection of the two fibers of $P(K_1,P_1)$ and $P(K_2,P_2)$, which is nonempty because $(\pi_2,\phi_2)$ is an epimorphism of principal bundles.  Meanwhile, it is clear that $F_{p_1}\cap F_{p_2}$ is isomorphic to $G_0$. This implies that $P/G_0$ is diffeomorphic to $P_1\times _M P_2$. Then to prove $P_C$ is a principal bundle, it suffices to  find a local section: Take open sets $U_i\subset P_i, i=1,2$ and let
    \[
    	U_1\times_M U_2:=U_1\times U_2\cap P_1\times_M P_2\subset P_1\times P_2,
    \]
    which is a general open set of the submanifold $P_1\times_M P_2$. Assume $U_i$ admits a local section $S_i:U_i\to P$. Since $S_1(U_1)K_1=\pi_1^{-1}(U_1)$ and $\pi_1^{-1}(U_1)\cap \pi_2^{-1}(U_2)$ is nonempty, one can choose $k_i\in K_i$ such that $U:=S_1(U_1)k_1\cap S_2(U_2)k_2$ is nonempty. By the assumption, for any given $(p_1,p_2)\in U_1\times_M U_2$, there exits $p\in U$ satisfying $(\pi_1(p),\pi_2(p))=(p_1,p_2)$ and the element $p$ is unique since the local section $S_i$ is injective. Hence the following map has to be a diffeomorphism from $U$ to its image
	\[
		(\pi_1,\pi_2):U\to P_1\times_M P_2: p\mapsto (\pi_1(p),\pi_2(p)),
	\]
	and the inverse is just the local section we want.
\end{proof}

\begin{Rm}
	As we saw in the commutative diagram \eqref{fig1}, $P$ can be viewed as a $G_0$ extension of the principal bundle $P_1\times_M P_2\to P_i$. When $P_1=P_2= Q$ and $G_0 = S^1$, the diagram
	\[
		\vcenter{\xymatrix @C=2.0pc @R=.7pc{
		& & Q \ar[dr] &\\
		S^1 \to P\, \ar[r] & Q\times_M Q \ar[dr] \ar[ur] & &M\\
		& & Q \ar[ur] &
		}}
	\]
 is similar to the one appeared in the definition of $S^1$-gerbes \cite{murray1996bundle}. More  general non-abelian gerbes were introduced in \cite{laurent2009non} as the groupoid extension. It is interesting to bring together these two kinds of extensions in some way.
 \end{Rm}

Now we introduce a special kind of DPBs with structure groups:
\begin{Def}
	A quadruple of Lie groups $\mathbb{G}=(G;G_1,G_2)_{G_0}$ is said to be a \emph{double Lie group} (DLG) with \emph{core} $G_0$ if it fits into the following exact sequence:
	\begin{equation}\label{DLG1}
		1\to G_0\xrightarrow{i} G \xrightarrow{\varphi} G_1\times G_2 \to 1.
	\end{equation}
\end{Def}

\begin{Rm}
	In \cite{lu},  a double Lie group denotes a triple $(G,G_+,G_-)$ of Lie groups where $G_+$ and $G_-$  are both Lie subgroups of $G$ such that the map $G_+\times G_-\to G$  given by $(g_+,g_-)\mapsto g_+g_-$ is a diffeomorphism. Here we adopt the same terminology  for a different notion.
\end{Rm}

Set $K_1=\varphi^{-1}(\{e\}\times G_2)$ and $K_2=\varphi^{-1}(G_1\times\{e\})$. The restrictions of $\varphi$ give the following two epimorphisms: 
\[
	\phi_2:=\varphi|_{K_1}: K_1\to G_2;\quad \phi_1:=\varphi|_{K_2}: K_2\to G_1.
\]
\begin{Lem}\label{Lem:2.6}
	The given $K_1$ and $K_2$ are closed normal subgroups  of $G$ such that $i(G_0)= K_1\cap K_2$ and the exact sequences follow:
	\[
		\vcenter{\xymatrix @C=2.3pc @R=1.5pc{
		& & K_1 \ar@{^{(}->}[d] \ar[r]^{\phi_2} & G_2 \ar@{^{(}->}[d] \ar[rd] &\\
		1 \ar[r] & G_0 \ar[ur] \ar[dr] \ar[r] &G \ar[r] &G_1\times G_2 \ar[r]  &1\\
		& & K_2 \ar@{^{(}->}[u] \ar[r]^{\phi_1}& G_1  \ar@{^{(}->}[u] \ar[ru]&
		}}.
	\]
	Moreover, $G=K_1K_2$ if $G$ is connected. 	
\end{Lem}
\begin{proof}
  By the exactness of the sequence \eqref{DLG1}, we get
  \[
  	i(G_0)=\ker \varphi=\varphi^{-1}(\{e\}\times G_2)\cap\varphi^{-1}(G_1\times\{e\})=K_1\cap K_2.
  \]
  Then the exact sequences are easy to get. Next we show  $G=K_1K_2$ if $G$ is connected: First $K_1K_2$ is a subgroup of $G$ since $K_1$ is a normal subgroup of $G$. To see $K_1K_2$ is a regular Lie subgroup, we need to prove that it is a closed submanifold. Take $g\notin K_1K_2$ and $U\subset K_2$ a small neighborhood of $e$. Then $gU$ is neighborhood of $g$ in $G$. It follows immediately that $gU\cap K_1K_2=\emptyset$, which means $K_1K_2$ is closed. Note that the composition $\sigma:K_1K_2\hookrightarrow G\to G_1$ is an epimorphism and $K_1=\ker\sigma$. Then we get $\dim K_1K_2= \dim K_1+\dim G_1=\dim G$.

  It suffices to show that the connected component of $K_1K_2$ is $G$. Choose a neighborhood $U\subset K_1K_2$ of $e\in K_1K_2$ under the sub-topology. From the dimension argument, $U$ must be an open subset in $G$ and then a neighborhood of $e\in G$. Since $G$ is connected, we have
  \[
  	G=\cup_{n=1}^{\infty}U^n\subset K_1K_2\subset G,
  \]
  where $ U^n$ consists of all $n$-fold products of elements of $U$. So we conclude that $G=K_1K_2$.
\end{proof}

We always assume that the group $G$ is either connected or $G=K_1K_2$. More details about double Lie groups will be discussed in Section \ref{DLGDLA}. Let $\mathbb{G}=(G;G_1,G_2)_{G_0}$ be a DLG.

\begin{Def}
	A DPB $\mathbb{P}=(P;P_1,P_2;M)$ is called a \emph{$\mathbb{G}$-DPB} if $P$ is a $G$-space such that both of the actions of $K_1$ and $K_2$ on $P$ are induced by the action: $G\to \widehat{G} \subset \Diff(P)$. In this case, the DLG $\mathbb{G}$ is called the \emph{structure group} of $\mathbb{P}$.
\end{Def}

It is known that a Lie group $G$ is $K$-principal bundle over the homogeneous space $G/K$ for any closed subgroup $K$, since there always exists a local section $G/K\to G$. Hence, what follows is clear:

\begin{Lem}
	A double Lie group $\mathbb{G}=(G;G_1,G_2)$ is a $\mathbb{G}$-DPB $(G;G_1,G_2;e)$:
	\begin{equation*}
		\vcenter{\xymatrix{
			G \ar[d]^{\varphi_1} \ar[r]^{\varphi_2} &G_2 \ar[d]\\
			G_1 \ar[r] &e
		}},
	\end{equation*}
	where $\varphi_i:G\to G_1\times G_2\to G_i$ is the composition of $\varphi$ and the projection.
\end{Lem}

One consequence of the free action of $G$ on $P$ is: 
\begin{Pro}\label{PoverM}
	The total space  $P$  of a $\mathbb{G}$-DPB  $\mathbb{P}=(P;P_1,P_2;M)$ is a principal $G$-bundle over $M$. 
\end{Pro}
\begin{proof}
	The proof consists in the construction of a local trivialization, which is equal to find a local section. Take $m\in M$ and its neighborhood $U$ which admits a local section $s_1: U\to P_1$. Then take a neighborhood $U_1\subset s_1(U)$ of $s_1(m)$ such that there exists a local section $S_1: U_1\to P$.	Hence $S_1\circ s_1: s_1^{-1}(U_1)\to P$ is the local section we want.
\end{proof}

\begin{Ex}
	Let $\pi_i:P_i\to M, i=0,1,2$ be three principal bundles with structure groups $G_i$ respectively. Take $P=P_1\times_M P_2\times_M P_0$ and $G=G_1\times G_2\times G_0$. Then it is easily seen that $(P;P_1,P_2;M)$ is a $\mathbb{G}$-DPB with $\mathbb{G}=(G;G_1,G_2)_{G_0}$, which is said to be \emph{trivial}. On the other hand, let $(P;P_1,P_2;M)$ be a $\mathbb{G}$-DPB, then $P$ can be decomposed as the fiber product of three principal bundles over $M$ if $P$ is a $\mathbb{G}$-DPB with $G=G_1\times G_2\times G_0$.
\end{Ex}

In general, it is not true that there always exits a DLG $\mathbb{G}$ for a DPB $\mathbb{P}$ such that $\mathbb{P}$ is a $\mathbb{G}$-DPB, {\em i.e.}, not every DPB admits a structure DLG. Especially, not every DPB over a point is a DLG. We now construct such an example.

Take three connected Lie groups $G_i$ for $i=0,1,2$. Assume that there are group homomorphisms $\psi_i:G_i\to \Aut(G_0)$ for $i=1,2$. Thus we get semi-direct products
\[
	K_1:=G_2\ltimes_{\psi_2} G_0; \quad K_2:=G_1\ltimes_{\psi_1} G_0.
\]
Let $P:=G_1\times G_2\times G_0$ be the Cartesian product of manifolds. Note that here we view $P$ as a manifold only. Clearly, $P=G_1\times K_1=G_2\times K_2$ as manifold. Thus $P_{I}:=G_1\times K_1$ and $P_{II}:=G_2\times K_2$ are trivial bundles over $G_1$ and $G_2$ respectively. Then we have

\begin{Pro}\label{Pro:2.12}
	Following the notation above, the quadruple $\mathbb{P}:=(P;G_1,G_2;\pt)$ is a DPB. Moreover, the DPB $\mathbb{P}$ has a structure of a DLG if and only if the actions of $G_1$ and $G_2$ on $G_0$ commute.
\end{Pro}
\begin{proof}
	The first statement follows from the definition.

	For the second part, on the one hand, if the actions of $G_i$'s commute, we can simply define $P$ as the group $(G_1\times G_2)\ltimes_{\psi_1\times \psi_2} G_0$. Note that there might be various group structures on $P$ such that $\mathbb{P}$ is a $\mathbb{P}$-DPB. 

	On the other hand, suppose that $\mathbb{P}$ is a DLG. Since $G_1$ is a subgroup of $K_2$, hence $G_1$ is a subgroup of $P$. Since $K_1$ is a normal subgroup of $P$ and $G_1\cap K_1=\{e\}$, thus one can show $P=G_1K_1$ in the same fashion as in Lemma \ref{Lem:2.6}. Hence $P$ must be a semi-direct product $G_1\ltimes_{\psi} K_1$ for some group homomorphism $\psi: G_1\to \Aut(K_1)$. Take $g_i\in G_i$ for $i=0,1,2$, we have
	\[
		(g_1,g_2,e)\cdot (e,e,g_0)=\left( g_1, g_2\pr_2(\psi(g_1)(e,g_0)),\psi_2(g_2)\left(\pr_0(\psi(g_1)(e,g_0))\right) \right),
	\]
	by the definition of the semi-direct product $P=G_1\ltimes_{\psi} (G_2\ltimes_{\psi_2} G_0)$, where $\pr_i:K_1\to G_i$ for $i=0,2$. Since $P/G_0\cong G_1\times G_2$ and $K_2$ is a normal subgroup of $P$, we have
	\[
		\pr_2(\psi(g_1)(e,g_0))=e, \quad \pr_0(\psi(g_1)(e,g_0))=\psi_1(g_1)(g_0).
	\]
	Thus we get
	\begin{equation}\label{eq:prop2.11(1)}
		(g_1,g_2,e)\cdot (e,e,g_0)=\big( g_1, g_2,\psi_2(g_2)\left(\psi_1(g_1)(g_0)\right) \big).
	\end{equation}

	Note that $G_2$ plays a similar role as $G_1$. By viewing $P=G_2\ltimes_{\psi'} K_2$ for some $\psi': G_2\to \Aut(K_2)$, similar calculation gives
	\begin{equation}\label{eq:prop2.11(2)}
		(g_1,g_2,e)\cdot (e,e,g_0)=\big( g_1, g_2,\psi_1(g_1)\left(\psi_2(g_2)(g_0)\right) \big).
	\end{equation}
	By the fact that $P$ is a group, Eq \eqref{eq:prop2.11(1)} must be equal to Eq \eqref{eq:prop2.11(2)}, which is equivalent to
	\[
		\psi_2(g_2)\left(\psi_1(g_1)(g_0)\right) =\psi_1(g_1)\left(\psi_2(g_2)(g_0)\right),
	\]
	{\em i.e.}, the actions of $G_i$'s for $i=1,2$ on $G_0$ commute.
\end{proof}

\begin{Rm}
	 Note that, as a square of Lie group bundles, the DPB $\mathbb{P}$ in Proposition \ref{Pro:2.12} carries a structure of a double Lie groupoid \cite{Mackenzie} over a point, whose source and target maps are natural projections. Hence this example clarifies three facts:
	 \begin{itemize}
	 \item [\rm (1)] There are DPBs that are not $\mathbb{G}$-DPBs;
	 \item [\rm (2)] A DPB over a point is not necessarily a DLG;
	 \item [\rm (3)] A double Lie groupoid over a point is not necessarily a DLG.
	 \end{itemize}
	 Besides, it follows by definition that a DLG is not necessarily a double Lie groupoid over a point.
\end{Rm}

As we saw above, homogeneous spaces of a Lie group provide basic examples for principal bundles. We now construct the \emph{homogeneous space} for a double Lie group, which  gives a class of  DPBs.
\begin{Pro}
Let $\mathbb{G}=(G;G_1,G_2)_{G_0}$ be  a double Lie group,  $ H \subset G$ a closed  subgroup, and 
$$ H_1:=H/H\cap K_1, \quad H_2:=H/H\cap K_2, \quad H_0:= H\cap G_0.$$
Then $\mathbb{H}:=(H;H_1,H_2)_{ H_0}$ is a double Lie group and $\mathbb{G}/\mathbb{H}$, defined below, is an  $\mathbb{H}$-DPB:
\begin{equation*}
		\mathbb{H}=\vcenter{\xymatrix{	
			H \ar[d]^{\varphi_1} \ar[r]^{\varphi_2} &H_2 \ar[d]\\
			H_1 \ar[r]&e
		}},\quad
		\mathbb{G}/\mathbb{H} :=
	\vcenter{\xymatrix {
			G \ar[d]\ar[r] &G/H\cap K_2 \ar[d]\\
				G/H\cap K_1 \ar[r] &G/H
		}},
	\end{equation*}
where $\varphi_1$ and $\varphi_2$ are the natural projections.
\end{Pro}
\begin{proof}
	Since $H\cap K_i\tl H$, we deduce the following group epimorphism:
	\[
		\varphi: H\to H_1\times H_2: h\mapsto (h(H\cap K_1),h(H\cap K_2)),
	\]
	which shows $\mathbb{H}$ is a DLG. According to Proposition 1.59 in \cite{felix2008algebraic}, the manifold $G/H\cap K_1$ is a principal $H_1$-bundle over $G/H$. All the other claims are easy to check.
\end{proof}

\section{More on Double Lie Groups}\label{DLGDLA}

In this section, we concentrate on the double Lie groups and show that they can be used to illustrate the automorphisms of DVSs. First let us review the double Lie groups we defined in Section \ref{DPFB}: A double Lie group consists of four Lie groups $\mathbb{G}=(G;G_1,G_2)_{G_0}$ satisfying the following exact sequence:
\begin{equation}\label{DLGrecall}
	1\to G_0\to G \xrightarrow{\varphi} G_1\times G_2 \to 1.
\end{equation}
\begin{Ex}
	Let $G_1$, $G_2$ and $G_0$ be Lie groups and let $G=G_1\times G_2\times G_0$. Then $(G;G_1,G_2)_{G_0}$ is a double Lie group, which will be said \emph{trivial}. Moreover, if there is a group homomorphism $\phi:G_1\times G_2\to \Aut(G_0)$, we conclude that $(G';G_1,G_2)_{G_0}$ is a double Lie group, where $G'=(G_1\times G_2)\ltimes_{\phi} G_0$ is the semi-direct product of $G_1\times G_2$ and $G_0$ with respect to $\phi$.
\end{Ex}

In fact, by  the exact sequence \eqref{DLGrecall},  we see $G$ is the extension of $G_1\times G_2$ by $G_0$. So the two examples above are nothing but two special extensions. Before the discussion of the automorphism group of a double vector space,  we first recall the definition of double vector bundles:

\begin{Def}\label{DVB}\emph{\cite{Mackenzie1992Double}}
	A \emph{double vector bundle} $\mathbb{E}=(E;E_1,E_2;M)$ is a system of vector bundles
        \begin{equation*}
		\vcenter{\xymatrix{			
		           E \ar[d]^{\sigma_1} \ar[r]^{\sigma_2} &E_2 \ar[d]^{\delta_2}\\
			E_1 \ar[r]^{\delta_1} &M
		}},\qquad
		\vcenter{\xymatrix {
			V \ar[d]^{\sigma_1} \ar[r]^{\sigma_2} &V_2 \ar[d]\\
			V_1 \ar[r] &0
		}},
	\end{equation*}
	in which $E$ has two vector bundle structures, on bases $E_1$ and $E_2$,  which are themselves vector bundles on $M$, such that each of the structure maps of each vector bundle structure on $E$ (the bundle projection, addition, scalar multiplication and the zero section) is a morphism of vector bundles with respect to the other structure. A double vector bundle $\mathbb{V}=(V;V_1,V_2)$ over a point will be called a \emph{double vector space} (DVS). A morphism of double vector bundles 
	\[
		(a;a_1,a_2;a_M): (E;E_1,E_2;M)\to (E';E'_1,E'_2;M')
	\]
	consists of maps $a:E\to E'$, $a_1:E_1\to E'_1$, $a_2:E_2\to E'_2$ and $a_M:M\to M'$ such that each of $(a,a_1)$, $(a,a_2)$, $(a_1,a_M)$ and $(a_2,a_M)$ is a morphism of the relevant vector bundles.
\end{Def}

To simplify notations, $E_I$ and $E_{II}$ are used for the vector bundles $E\to E_1$ and $E\to E_2$ respectively. Correspondingly, use $0_I$ and $0_{II}$  to denote their zero sections. The zero section of $E_i$ is denoted by $0_i$. And $\sigma: E\to M$ is the following composition $\delta_1\circ \sigma_1=\delta_2\circ \sigma_2$. In fact, there is a third vector bundle $E_0$ on $M$ besides $E_1$ and $E_2$, known as the \emph{core}, which is an embedded submanifold of $E$:
\[
	E_0=\{u\in E\,|\,\exists m\in M \mbox{~such that~} \sigma_1(u)=0_1(m), \sigma_2(u)=0_2(m)\}.
\]

For a DVS $\mathbb{V}=(V;V_1,V_2)_{V_0}$, let $\Aut(\mathbb{V})$ be the automorphism group of $\mathbb{V}$,  which is consist of invertible morphisms of $\mathbb{V}$ and turns out to be a DLG. Denote
\[
	\Aut(V):=\{a\,|\,(a;a_1,a_2)\text{~is an automorphism of~} \mathbb{V}\}.
\]
\begin{Thm}\label{thm Aut} 
	Let $\mathbb{V}=(V;V_1,V_2)_{V_0}$ be a DVS.
	Then one has
	\[
		 \Aut(V)=\Aut(V_I)\cap \Aut(V_{II}) \cong (\GL(V_1)\times \GL(V_2))\ltimes (\GL(V_0)\ltimes T),
	\]
	where $T=\Hom (V_1\otimes V_2, V_0)$. Consequently, one gets the double Lie group:
	\begin{equation}\label{DL}
		\Aut(\mathbb{V}): \quad 1\to \GL(V_0)\ltimes T \to \Aut(V) \to \GL(V_1)\times \GL(V_2) \to 1.
 	\end{equation}
\end{Thm}
\begin{proof}
	First let us prove this theorem for the trivial double vector space $\mathbb{V}$. We will write down the group $\Aut(V)$ explicitly in this case. Let $(a;a_1,a_2)$ be an automorphism of $\mathbb{V}$. Then  $a_1$ and $a_2$ are linear isomorphisms since $V_1$ and $V_2$ are vector spaces. Unravel $a$ as
	\begin{align*}
		a=(\psi_1,\psi_2,\psi_0): V_1\oplus V_2\oplus V_0 &\to V_1\oplus V_2\oplus V_0\\
		(v_1,v_2,v_0)&\mapsto (\psi_1(v_1,v_2,v_0),\psi_2(v_1,v_2,v_0),\psi_0(v_1,v_2,v_0)).
	\end{align*}
	An analysis of $\psi_i$ will show the structure of $\Aut(V)$. Since $(a,a_1)$ is a vector bundle isomorphism,  we know $a_1\circ\sigma_1=\sigma_1\circ a$ and $a\circ 0_I=0_I\circ a_1$, which imply $\psi_1(v_1,v_2,v_0)=a_1(v_1)$ and $\psi_0(v_1,0,0)=0$. Similarly, we have $\psi_2(v_1,v_2,v_0)=a_2(v_2)$ and $\psi_0(0,v_2,0)=0$. To prove that $\psi_0$ can be decomposed into the sum of a linear map and a bilinear map, let
	\[
		a_0(v_0):=\psi_0(0,0,v_0),\qquad \mu(v_1,v_2):=\psi_0(v_1,v_2,0).
	\]
	Direct computation gives
		\begin{align*}
			\psi_0(v_1,v_2,v_0)&=\psi_0(v_1,0,v_0)+_1\psi_0(v_1,v_2,0)=(\psi_0(v_1,0,0)+_2\psi_0(0,0,v_0))+_1\psi_0(v_1,v_2,0)\\
			&=a_0(v_0)+\mu(v_1,v_2).
		\end{align*}
	Thus $a$ can be written as $(a_1,a_2,a_0,\mu)$, where $\mu\in T$, such that 
	$$(a_1,a_2,a_0,\mu)	(v_1,v_2,v_0)=
	(a_1(v_1), \,a_2(v_2), \,a_0(v_0)+\mu(v_1,v_2)).$$
	
	 It is clear that the composition of $a$ and $b$ is the product as follows:
		\begin{equation}\label{Auto}
			(a_1,a_2,a_0,\mu)\cdot(b_1,b_2,b_0,\nu)=(a_1b_1,a_2b_2,a_0b_0,\mu\circ(b_1\times b_2)+a_0\circ\nu).
		\end{equation}	
	On the other hand, any quadruple $(a_1,a_2,a_0,\mu)$ induces an automorphism of $\mathbb{V}$. Thus we proved that $\Aut(V)=(\GL(V_1)\times \GL(V_2))\ltimes (\GL(V_0)\ltimes T)$. By definition, we have
	\[
		\Aut(V)\subset \Aut(V_I)\cap \Aut(V_{II}).
	\]
	On the other hand, for $f\in \Aut(V_I)\cap \Aut(V_{II})$, the induced map of $f$ on $V_2$ is linear since $f$  is linear on the fiber $(v_1,V_2\oplus V_0)$. Thus we get the inverse inclusion $\Aut(V)\supset \Aut(V_I)\cap \Aut(V_{II})$.
	To show the exact sequence \eqref{DL}, we define the group homomorphism:
	\[
		\varphi: \Aut(V)\to \GL(V_1)\times \GL(V_2): a\mapsto (\sigma_1\circ a\circ 0_I,\sigma_2\circ a\circ 0_{II}),
	\]
	where $0_I: V_1\to V$ and $0_{II}: V_2\to V$ are the zero sections. This ends the proof of this theorem for the trivial DVS.

	For a general double vector space $\mathbb{V}=(V;V_1,V_2)_{V_0}$, there is a diffeomorphism $(\Psi;\Psi_1,\Psi_2)$ from the trivial double vector space $V_1\oplus V_2\oplus V_0$ to $\mathbb{V}$ (see e.g. Proposition 2.11 in \cite{gracia2010lie}). Hence $\Psi$ induces an isomorphism of double Lie groups $\Aut(\mathbb{V})\cong\Aut(V_1\oplus V_2\oplus V_0)$ by
	\[
		a \mapsto \Psi^{-1}\circ a\circ\Psi,\quad a_i \mapsto \Psi_i^{-1}\circ a_i\circ\Psi_i,
	\]
	where $(a;a_1,a_2)$ is an automorphism of $\mathbb{V}$. It is direct to show the other results for $\mathbb{V}$.
\end{proof}

Just like representations of Lie groups, we have:
\begin{Def}
	A \emph{morphism} of DLGs $(\rho;\rho_1,\rho_2)_{\rho_0}: (G;G_1,G_2)_{G_0}\to (G';G_1',G_2')_{G'_0}$ consists of a system of Lie group homomorphisms: $\rho:G\to G'$ and $\rho_i:G_i\to G_i'$ for $i=0,1,2$:
	\begin{equation*}
			\vcenter{\xymatrix @ C=1.2pc @R=1.3pc{
		1\ar[r] & G_0\ar[r] \ar[d]^{\rho_0} &G \ar[r] \ar[d]^{\rho}&G_1\times G_2 \ar[r] \ar[d]^ {\rho_1\times \rho_2} & 1\\
		1\ar[r] & G'_0\ar[r]  &G' \ar[r]  &G'_1\times G'_2 \ar[r] & 1
		}}
	\end{equation*}
	such that the relevant diagrams commute. A \emph{representation} of $\mathbb{G}$ consists of a double vector space $\mathbb{V}=(V;V_1,V_2)_{V_0}$ and a double Lie group morphism $(\rho;\rho_1,\rho_2)_{\rho_0}$ from $\mathbb{G}$ to $\Aut(\mathbb{V})$.
\end{Def}

 For a vector space with a Lie group representation, its dual space is equipped with the dual representation. So it is natural to consider the dual representation of a DLG on a DVS. Let $\mathbb{V}=(V;V_1,V_2)_{V_0}$ be a DVS. Then its dual with respect to $V_1$ (\cite{konieczna1999double,Mackenzie1999}) is $\mathbb{V}_I^*=(V^*_I,V_1,V_0^*)_{V_2^*}$:
      \begin{equation*}
		\vcenter{\xymatrix @C=.6pc @R=.6pc{	
		V \ar[dd] \ar[rr] & &V_2 \ar[dd]\\
		 & V_0 \ar[dr]&\\
		V_1 \ar[rr] & &0
	}}\Longrightarrow
	\vcenter{\xymatrix @C=.6pc @R=.6pc{
		V^*_I \ar[dd] \ar[rr] & &V^*_0 \ar[dd]\\
		 & V^*_2 \ar[dr]&\\
		V_1 \ar[rr] & &0
	}},
	\end{equation*}
where $V^*_I\to V_1$ is the dual vector bundle of $V_I$. The duality interchanges the positions of $V_2$ and $V_0$. We get two double Lie groups $\Aut(\mathbb{V})$ and $\Aut(\mathbb{V}_I^*)$ as follows:
\begin{equation}\label{automorphism group}
	\vcenter{\xymatrix {		
		\Aut(V) \ar[d]^{\psi_1} \ar[r]^{\psi_2} &\GL(V_2) \ar[d]\\
		\GL(V_1) \ar[r]&e
	}},\quad
	\vcenter{\xymatrix{
		\Aut(V^*_I) \ar[d] \ar[r] &\GL(V^*_0) \ar[d]\\
		\GL(V_1) \ar[r]&e
	}}.
\end{equation}
And the corresponding exact sequence for the DLG $\Aut(\mathbb{V}_I^*)$ is 
\begin{equation}\label{autofstar}
	1\to G_0^*\to \Aut(V_I^*) \to \GL(V_1)\times \GL(V_0^*)\to 1,
\end{equation}
where the core
\[
	G_0^*=\GL(V_2^*)\ltimes T_I^*,  \quad   T_I^*=\Hom(V_1\otimes V_0^*,V_2^*).
\]

\begin{Pro}\label{Dual}
	There is a canonical anti-isomorphism between the Lie groups $\Aut(V)$ and $\Aut(V_{I}^*)$.
\end{Pro}
\begin{proof}
	Fix a decomposition of DVS $\mathbb{V}$: $\Psi: V_1\oplus V_2\oplus V_0 \to \mathbb{V}$. Then there is a natural decomposition $\Psi_I^*:V_1\oplus V_0^*\oplus V_2^* \to \mathbb{V}_I^*$  (see \cite{notes}).
	Define the following map:
	\[
		f:\Aut(V)\to \Aut(V_I^*):\Psi\circ(a_1,a_2,a_0,\mu)\circ\Psi^{-1}\mapsto \Psi_I\circ (a_1^{-1},a_0^*,a_2^*,\mu^*_I\circ a_1^{-1})\circ\Psi_I^{*-1},
	\]
	where $a_i^*$ is the transposition of $a_i$ and $\mu^*_I: V_1\times V_0^*\to V_2^*$ is the dual of $\mu$ with respect to $V_1$
	\[
		\langle\mu^*_I(v_1,\eta_0),v_2\rangle:=\langle\eta_0,\mu(v_1,v_2)\rangle.
	\] 
	For convenience, the decomposition $\Psi$ will be omitted in the following.

	We check that $f$ is an (anti-)isomorphism between $\Aut(V)$ and $\Aut(V^*_{I})$. Obviously, the map $f$ is a bijection. It suffices to prove $f$ is an (anti-)morphism, namely,
	\begin{equation}\label{homo}
		f\big((a_1,a_2,a_0,\mu)(b_1,b_2,b_0,\nu)\big)=f(b_1,b_2,b_0,\nu)f(a_1,a_2,a_0,\mu).
	\end{equation}
	By the group multiplication \eqref{Auto}, the left hand side of \eqref{homo} equals to
	\[
		\big((a_1b_1)^{-1},(a_0b_0)^*,(a_2b_2)^*,(a_0\circ\nu+\mu\circ(b_1\times b_2))_I^*\circ (a_1b_1)^{-1}\big),
	\]
	and the right hand side is
	\[
		\big(b_1^{-1}a_1^{-1},b_0^*a_0^*,b_2^*a_2^*,b_2^*\circ(\mu^*_I\circ a_1^{-1})+\nu^*_I\circ b_1^{-1}(a_1^{-1}\times a_0^*)\big).
	\]
	It is obvious that the first three terms are the same. For the last term,  acting on an element $(v_1,\eta_0)\in V_1\times V_0^*$ and taking pairing with $v_2\in V_2$, we find that both of them are equal to
	\[
		\langle \eta_0,a_0\nu((a_1b_1)^{-1}v_1,v_2)+\mu(a_1^{-1}v_1,b_2v_2)\rangle.
	\]
	This proves that $f$ is an (anti-)isomorphism.
\end{proof}

Thus, an $\Aut(V)$-space is also an $\Aut(V_I^*)$-space. So we have:
\begin{Cor}\label{Dual DPB}
	Suppose $\mathbb{V}=(V;V_1,V_2)$ is a DVS and $\mathbb{P}=(P;P_1,P_2;M)$ is an $\Aut(\mathbb{V})$-DPB. Then we have the following $\Aut(\mathbb{V}_I^*)$-DPB $\mathbb{P}_I^*:=(P;P_1,Q_2;M)$:
	\[
	\mathbb{P}_I^*:=\quad
		\vcenter{\xymatrix{
		P \ar[d] \ar[r] & Q_2 \ar[d]\\
			P_1 \ar[r]&M
		}}, \quad Q_2:=P/\GL(V_1)\ltimes G_0^*.
	\]
\end{Cor}
\begin{proof}
	By the exact sequence \eqref{autofstar}, let $K_1':=\GL(V_0^*)\ltimes G_0^*$ and $K_2':=\GL(V_1)\ltimes G_0^*$. Note that the isomorphism $f$ also induces an isomorphism between $K_1$ and $K_1'$. Thus, the $\Aut(V)$-space $P$ is also an $\Aut(V_I^*)$-space with the action $pg^*:=pf^{-1}(g^*)$ and the quotient $P/K_1'$ is diffeomorphic to $P_1$. By Proposition \ref{PoverM}, $P$ is a principal $\Aut(V_I^*)$-bundle over $M$. Suppose $U\times \Aut(V_I^*)$ is a local trivialization of the bundle $P\to M$, where $U$ is an open subset in $M$. Then $Q_2$ becomes a smooth manifold by taking the local chart $U\times \GL(V_0^*)$. Thus it is clear now that $P\to Q_2$ and $Q_2\to M$ are principal bundles. The rest is easy to check.
\end{proof}

As the isomorphism $f$ in Proposition \ref{Dual} does not induce an isomorphism between $\GL(V_2)$ and $\GL(V_0^*)$, the two DLGs $\Aut(\mathbb{V})$ and $\Aut(\mathbb{V}^*_I)$ are not isomorphic. So there is not a natural dual representation of a DLG $\mathbb{G}$ on the dual DVS $\mathbb{V}^*_{I}$.

The local structure  of a double Lie group $\mathbb{G}$ is characterized by  its Lie algebra  $\Lie(\mathbb{G})$:
	\begin{equation}\label{dla}
	0\to \mathfrak{g}_0\to \mathfrak{g} \to \mathfrak{g}_1\oplus \mathfrak{g}_2 \to 0.
	\end{equation}
We call a triple of Lie algebras $\widetilde{\mathfrak{g}}=(\mathfrak{g};\mathfrak{g}_1,\mathfrak{g}_2)$ satisfying exact sequence \eqref{dla} a
\emph{double Lie algebra} with  core   $\mathfrak{g}_0$.  Note that there is  not a natural DVS  structure  for $\widetilde{\mathfrak{g}}$ with the commutative diagram
\[
	\vcenter{\xymatrix{
		\mathfrak{g} \ar[d]^{\tau_1} \ar[r]^{\tau_2} &\mathfrak{g}_2 \ar[d]\\
		\mathfrak{g}_1 \ar[r]&0
	}},	
\]
since  $\mathfrak{g}_i$ is a quotient space 
instead of a subspace of $\mathfrak{g}$  by definition.

From \eqref{DL}, it is not hard to see that the Lie algebra $\mathfrak{aut}(\mathbb{V})$ of the DLG $\Aut(\mathbb{V})$
is given by 
\begin{equation}\label{dl}
	\mathfrak{aut}(\mathbb{V}): \quad 0\to \mathfrak{gl}(V_0)\ltimes T \to \mathfrak{aut}(V) \xrightarrow{\varphi} \mathfrak{gl}(V_1)\oplus \mathfrak{gl}(V_2) \to 0,
\end{equation}
where
\[
	\mathfrak{aut}(V)=\{(A_1,A_2,A_0,\alpha)\,|\,A_i\in  \mathfrak{gl}(V_i),\alpha\in
	T=\Hom (V_1\otimes V_2, V_0)\}.
\]
Actually, we have another characterization of $\mathfrak{aut}(V)$.
\begin{Pro}
	The Lie algebra $\mathfrak{aut}(V)$  is the semi-direct product of  $\mathfrak{gl}(V_1)\oplus \mathfrak{gl}(V_2)\oplus \mathfrak{gl}(V_0)$ and $T$ with the action given by
	\[
		(A_1,A_2,A_0)\triangleright \nu=A_0\circ \nu-\nu\circ(A_1\times I)-\nu\circ (I\times A_2).
	\]
\end{Pro}
\begin{proof}
	The Lie bracket of $\mathfrak{aut}(V)$ now is
	\begin{eqnarray*}
		[(A_1,A_2,A_0,\mu),(B_1,B_2,B_0,\nu)]=\big([A_1,B_1],[A_2,B_2],[A_0,B_0],\Upsilon\big),
	\end{eqnarray*}
	where
	\[
		\Upsilon=A_0\circ \nu-\nu\circ(A_1\times I)-\nu\circ (I\times A_2)-B_0\circ \mu+\mu\circ(B_1\times I)+\mu\circ(I\times B_2).
	\]
	We check the bracket of $(A_1,A_2,A_0,0)$ and $(0,0,0,\nu)$. Note that the exponential map $\exp:\mathfrak{aut}(V)\rightarrow \Aut(V)$ is
	\[
		\exp t(A_1,A_2,A_0,0)=(e^{tA_1},e^{tA_2},e^{tA_0},0),\qquad \exp t(0,0,0,\nu)=(I,I,I,t\nu).
	\]
	Thus, by the relation of the Lie bracket and the group structure, we get
	\begin{align*}
		&[(A_1,A_2,A_0,0),(0,0,0,\nu)]\\
		=&\frac{d}{dt}\frac{d}{ds}\Big|_{t,s=0} (e^{tA_1},e^{tA_2},e^{tA_0},0)(I,I,I,s\nu)(e^{-tA_1},e^{-tA_2},e^{-tA_0},0)(I,I,I,-s\nu)\\
		=&\frac{d}{dt}\frac{d}{ds}\Big|_{t,s=0} \Big(I,I,I,-s\nu+(e^{tA_0} \circ s\nu)(e^{-tA_1}\times e^{-tA_2})\Big)\\
		=&\big(0,0,0,A_0\circ \nu-\nu\circ (A_1\times I)-\nu\circ (I\times A_2)\big).
	\end{align*}
	This ends the proof.
\end{proof}

In general, a double Lie algebra $(\mathfrak{g};\mathfrak{g}_1,\mathfrak{g}_2)$ is a non-abelian extension of  $\mathfrak{g}_1\oplus \mathfrak{g}_2$ by $\mathfrak{g}_0$. In the rest of this section, we discuss a simple case that the core $\mathfrak{g}_0$ is an abelian Lie algebra. 

Let $\mathfrak{m}$ be a ($\mathfrak{g}_1\oplus \mathfrak{g}_2$)-module with a representation $\rho =\rho_1 + \rho_2 :\mathfrak{g}_1\oplus  \mathfrak{g}_2 \to \mathfrak{gl}(\mathfrak{m})$. Obviously, $\rho_1:\mathfrak{g}_1\to  \mathfrak{gl}(\mathfrak{m})$ and $\rho_2:\mathfrak{g}_2\to \mathfrak{gl}(\mathfrak{m})$ are representations of $\mathfrak{g}_1$ and $\mathfrak{g}_2$ respectively and satisfy $[\rho_1(\cdot),\rho_2(\cdot)]=0$. Every $2$-cochain $\omega\in C^2(\mathfrak{g}_1\oplus \mathfrak{g}_2, \mathfrak{m})$ can be split into three parts:
\[
	\omega=\omega_{2,0}+\omega_{1,1}+\omega_{0,2}\in \Hom(\wedge^2 \mathfrak{g}_1,\mathfrak{m})\oplus \Hom(\mathfrak{g}_1\wedge \mathfrak{g}_2,\mathfrak{m})\oplus \Hom(\wedge^2 \mathfrak{g}_2,\mathfrak{m}).
\]
Then, we have:
\begin{Lem}
	The 2-cochain $\omega$ is a 2-cocycle with respect to the $\mathfrak{g}_1\oplus \mathfrak{g}_2$-module if and only if $\omega_{2,0}$ and $\omega_{0,2}$ are 2-cocycles with respect to the $\mathfrak{g}_1$-module and $\mathfrak{g}_2$-module respectively and the following compatibility conditions: 
	\[
		d_{\mathfrak{g}_1}(\iota_a \omega_{1,1})=\rho_2(a)\circ \omega_{2,0},\qquad  d_{\mathfrak{g}_2}(\iota_x \omega_{1,1})=\rho_1(x)\circ \omega_{0,2},\qquad \forall a\in \mathfrak{g}_2,x\in \mathfrak{g}_1,
	\]
	where $\iota$ denotes the contraction.
\end{Lem}
\begin{Ex}
    In particular, suppose that $\mathfrak{m}=\mathfrak{m}_1\oplus \mathfrak{m}_2$ with $\mathfrak{m}_1$ and $\mathfrak{m}_2$ being $\mathfrak{g}_1$ and $\mathfrak{g}_2$-modules respectively. So $\mathfrak{m}$ is naturally a $\mathfrak{g}$-module. Let $\omega_{2,0}\in \Hom(\wedge^2 \mathfrak{g}_1,\mathfrak{m}_1)$ and $\omega_{0,2}\in \Hom(\wedge^2 \mathfrak{g}_2,\mathfrak{m}_2)$ be two 2-cocycles and let $\theta_1\in \Hom(\mathfrak{g}_1,\mathfrak{m}_1)$ and $\theta_2\in \Hom(\mathfrak{g}_2,\mathfrak{m}_2)$ be two 1-cocycles relative to the corresponding $\mathfrak{g}_i$-module structure on $\mathfrak{m}_i$. Then the 2-cochian
    \[\omega=\omega_{2,0}+\theta_1\wedge \theta_2+\omega_{0,2}\in C^2(\mathfrak{g}_1\oplus \mathfrak{g}_2,\mathfrak{m}) \]
    is a 2-cocycle with respect to the $\mathfrak{g}_1\oplus \mathfrak{g}_2$-module structure on $\mathfrak{m}$.
\end{Ex}
\begin{Pro}
	Given a 2-cocycle  $\omega\in C^2(\mathfrak{g}_1\oplus \mathfrak{g}_2, \mathfrak{m})$, define a bracket on the
	direct sum space  $\mathfrak{g}=\mathfrak{g}_1\oplus \mathfrak{g}_2 \oplus  \mathfrak{m}$ as follows:
	\begin{align*}
		&[x+a+u,y+b+v]_{\mathfrak{g}}\\
		=&[x,y]_{\mathfrak{g}_1}+[a,b]_{\mathfrak{g}_2}+\omega(x+a,y+b)\\
		=&[x,y]_{\mathfrak{g}_1}+[a,b]_{\mathfrak{g}_2}+\omega_{2,0}(x,y)+\omega_{1,1}(x,b)-\omega_{1,1}(y,a)+\omega_{0,2}(a,b)
	\end{align*}
	Then  $(\mathfrak{g},\mathfrak{g}_1,\mathfrak{g}_2)$ forms a double Lie algebra with abelian core $\mathfrak{m}$.
\end{Pro}

It is seen that  $\omega_{2,0}$ and $\omega_{0,2}$ give the abelian extensions of $\mathfrak{g}_1$ and $\mathfrak{g}_2$ respectively.  $\omega_{1,1}$ gives a nontrivial bracket between $\mathfrak{g}_1$ and  $\mathfrak{g}_2$.

\section{Associated Double Vector Bundles}\label{ADVB}
In this section, we present our main result of this paper: A double vector bundle can be realized as an associated bundle of a $\mathbb{G}$-DPB, with the fiber being a double vector space. This result is formulated by two theorems. First, we prove the associated bundle of a DPB is a DVB. Then we define the frame bundle of a DVB, proved to be a DPB, whose associated bundle with respect to the trivial DVS is naturally isomorphic to the previous DVB.  Another interesting observation is that the dual double vector bundle can be obtained easily just by the dual representation of the automorphism group of a double vector space. 

Usually for a principal bundle $P(G,M)$ and a $G$-manifold $F$, the associated bundle of $P$ with fiber type $F$ is defined by
\[
	P\times_G F=\{[p,v]\,|\,(p,v)\sim (pg,g^{-1}v),\forall (p,v)\in P\times F, g\in G\}.
\]
In the case that $F$ is a vector space and the $G$-action is linear, the bundle $P\times_G F$ is a vector bundle on $M$. It is known that any vector bundle can be realized as an associated bundle of a principal bundle, e.g., its frame bundle.

\begin{Thm}\label{main}
	Let $\mathbb{P}$  be a  $\mathbb{G}$-DPB and	$\mathbb{V}$ a $\mathbb{G}$-module given by
	\[
		\vcenter{\xymatrix{
			P \ar[d]^{\pi_1} \ar[r]^{\pi_2} &P_2 \ar[d]^{\lambda_2}\\
			P_1 \ar[r]^{\lambda_1} &M
		}},\quad
		\vcenter{\xymatrix{
			G \ar[d]^{\varphi_1} \ar[r]^{\varphi_2} &G_2 \ar[d]\\
			G_1 \ar[r] &e
		}},\quad
		\vcenter{\xymatrix{
			V \ar[d]^{\sigma_1} \ar[r]^{\sigma_2} &V_2 \ar[d]\\
			V_1 \ar[r]&0
		}}.
	\]	
	Then one can get an associated DVB $\mathbb{P}\times_{\mathbb{G}} \mathbb{V}$ as follows: 
	\[
		\vcenter{\xymatrix{
			P\times_G V \ar[d]^{\Pi_1} \ar[r]^{\Pi_2} &P_2\times_{G_2}  V_2 \ar[d]^{\Lambda_2} & \\
			P_1\times_{G_1}V_1 \ar[r]^{\Lambda_1} &M & P\times _{G} V_0\ar[l]^{\Lambda_0}
		}},
	\]
	where $\pi:=\lambda_1\circ\pi_1=\lambda_2\circ\pi_2$,
	\[
		\Pi_i([p,v])=[\pi_ip,\sigma_iv];  \quad \Lambda_i([p_i,v_i])=\lambda_i(p_i); \quad \Lambda_0([p,v_0])=\pi(p), \quad (i=1,2).
	\]
\end{Thm}

Use $+_i$ and $\cdot_i$ for $i=1,2$ to denote the two different vector bundle structures on $V$ and $(\rho;\rho_1,\rho_2)$  to denote the representation of $\mathbb{G}$ on $\mathbb{V}$.  We will make the assumption that $G$ is either connected or $G=K_1K_2$. 
\begin{Lem}
	The projections $\Pi_i \, (i=1,2)$ are well-defined.
\end{Lem}
\begin{proof}
	We shall show that the definition of $\Pi_1$ only depends on the equivalence class. In other words, we have to prove $[\pi_1(p),\sigma_1(v)]=[\pi_1(pg),\sigma_1(g^{-1}v)]$ for $[p,v]=[pg,g^{-1}v]$. By definition, $\pi_1$ intertwines the actions of $G$ and $G_1$:  Since $G=K_1K_2$ as we assumed, for $g\in G$, there exist $k_1\in K_1$ and $k_2\in K_2$ such that $g=k_1k_2$, then
	\[
		\pi_1(pg)=\pi_1(pk_1k_2)=\pi_1(pk_1)\varphi_1(k_2)=\pi_1(p)\varphi_1(g),
	\]
	As $g^{-1}v=\rho(g^{-1})v$ and by definition of a representation for a DVB, we have
	\[
		\sigma_1(\rho(g^{-1})v)=\psi_1(\rho(g^{-1}))\sigma_1(v), \quad \psi_1(\rho(g^{-1}))=\rho_1(\varphi_1(g^{-1})),
	\]
	where the group morphism $\psi_1$ is defined in the diagram \eqref{automorphism group}. Thus $\sigma_1(g^{-1}v)=\varphi_1(g^{-1})\sigma_1(v)$, which means $\pi_1$ is well-defined.
\end{proof}
\begin{Lem}
    	$P\times_G V\xrightarrow{\Pi_1} P_1\times_{G_1} V_1$ is a vector bundle.
\end{Lem}
\begin{proof}
	We split the proof into three steps: 
	
	Step 1: Define a linear structure on the fiber of $[p_1,v_1]$. If $\Pi_1([p,v])=\Pi_1([q,u])$, then there $\exists! g\in G$ such that $q=pg$ because of the principal bundle structure of $P\to M$. Then we define the addition as:
	\begin{equation}\label{plus}
		[p,v]+_1 [q,u]:=[p,v+_1 gu].
	\end{equation}
	The scalar multiplication is obvious: For $r\in \mathrm{R}$, define $r \cdot_1 [p,v]:=[p,r\cdot_1 v]$.

	Then by definition, $[p,v]+_1 [q,u]=[p,v+_1 gu]=[q,g^{-1}v+_1u]=[q,u]+_1[p,v]$, which shows that $+_1$ is symmetric. Meanwhile, if $\Pi_1([p,v])=\Pi_1([q,u])=\Pi_1([p',v'])$ and $p'=qh$, then we obtain:
	\begin{align*}
		([p,v]+_1 [q,u])+_1[p',v']&=[q,g^{-1}v+_1u]+_1[p',v']=[q,g^{-1}v+_1u+_1hv']\\
		&=[p,v]+_1[q,u+_1hv']
		=[p,v]+_1([q,u]+_1[p',v']),
	\end{align*}
	which means $+_1$ is associative. This gives a linear structure on the fiber of $[p_1,v_1]$.

	Step 2: Prove the existence of the zero section. Let $0_I$ denote the zero section of $V\to V_1$. Define the zero section:
	\[
		\widehat{0}_I: P_1\times_{G_1} V_1\to P\times_G V: [p_1,v_1]\mapsto [p,0_I(v_1)],
	\]
	where $p\in P$ such that $\pi_1(p)=p_1$. First, it is clear that $[p,0_I(v_1)]$ is independent of the choice of $p$. Then we need to show $\widehat{0}_I$ is well-defined: For $[p_1,v_1]=[p_1g_1,g_1^{-1}v_1]$, choose $g\in G$ such that $\varphi_1(g)=g_1$ and $p\in P$ such that $\pi_1(p)=p_1$. Then we have $\pi_1(pg)=p_1g_1$ and $g^{-1}0_I(v_1)=0_I(g_1^{-1}v_1)$. Hence,
	$\widehat{0}_I([p_1g_1,g_1^{-1}v_1])=[pg,0_I(g_1^{-1}v_1)]=\widehat{0}_I([p_1,v_1])$.
	At last, we deduce $[p,v]+_1 [q,u]=[q,u]$ if and only if $[p,v]\in \im\widehat{0}_I$, which gives the uniqueness of the zero point on each fiber.
	
	Step 3: Show the existence of a local trivialization. Since $P_1$ and $P$ are principal bundles over $M$, there exists a neighborhood $U$ of $m\in M$ such that:
	\[
		\alpha_1: \lambda_1^{-1}(U)\xrightarrow{\sim} U\times G_1,\quad \alpha:\pi^{-1}(U)\xrightarrow{\sim} U\times G.
	\]
	Use $\beta_1$ to denote the composition of $\alpha_1$ and the projection from $U\times G_1$ to $G_1$. The map $\beta$ denotes the composition of $\alpha$ and the projection to $G$. Thus the following two maps are well-defined and 1-1:
	\begin{align*}
		\gamma_1:P_1\times_{G_1} V_1|_U\xrightarrow{\sim} U\times V_1&: [p_1,v_1]\mapsto (\lambda_1(p_1),\beta_1(p_1)v_1),\\
		\gamma:P\times_G V|_U\xrightarrow{\sim} U\times V&: [p,v]\mapsto (\pi(p),\beta(p)v).
	\end{align*}
	Take the open subset $W:=\gamma_1^{-1}(U\times V_1)$. Coupled with a trivialization of the vector bundle $V_I$ that $(\sigma_1,\tau):V|_{V_1}\xrightarrow{\sim} V_1\times \mathrm{R}^{n_1}$, we get a local trivialization of $P\times_G V\xrightarrow{\Pi_1} P_1\times_{G_1} V_1$ given by
	\[
		P\times_G V|_{W}\xrightarrow{\sim} W\times \mathrm{R}^{n_1},[p,v]\to \Big(\gamma_1^{-1}\Big(\pi(p),\sigma_1(\beta(p)v)\Big),\tau(\beta(p)v)\Big).
	\]
	Consequently, $P\times_G V$ is a vector bundle over $P_1\times_{G_1} V_1.$	
\end{proof}

\begin{proof}[Proof of Theorem \ref{main}]
	By Proposition 2.1 in \cite{gracia2010lie}, the only point remaining concerns  the interchange law. Without loss of generality, take four points in $P\times_G V$:
	\[
		[p,v],\quad[pk_1,u],\quad[pk_2,v'],\quad[pk_2k'_1,u'],
	\]
	for $k_1,k'_1\in K_1$ and $k_2\in K_2$ such that $(v,u),(v',u')\in V\times_{V_1} V$ and $(v,v'),(k_1u,k'_1u')\in V\times_{V_2} V$. Using the interchange law in the DVS $\mathbb{V}$, we have
	\begin{align*}
		&([p,v]+_1[pk_1,u])+_2([pk_2,v']+_1[pk_2k'_1,u'])=[p,(v+_1k_1u)+_2(k_2v'+_1k_2k'_1u')]\\
		=&[p,(v+_2k_2v')+_1(k_1u+_2k_2k'_1u')]=([p,v]+_2[pk_2,v'])+_1([pk_1,u]+_2[pk_2k'_1,u']).
	\end{align*}
	Then the interchange law holds. It is direct to see the core of the DVB $\mathbb{P}\times_{\mathbb{G}} \mathbb{V}$ is $P\times _G V_0$.
\end{proof}
From the above theorem, we get a DVB with core the vector bundle $P\times _G V_0$ on $M$. Recall that the core $P_C:= P\rightarrow P_1\times _M P_2$ of a DPB $\mathbb{P}$ is a principal $G_0$-bundle. As $G_0$ acts on $V_0$, we get its associated vector bundle $P\times _{G_0} V_0$ on $P_1\times _M P_2$. Using the natural map $\lambda_0=(\lambda_1,\lambda_2):P_1\times_M P_2\rightarrow M$, we establish the relation between these two vector bundles.

\begin{Pro}
    With notations above, 
    we have $P\times _{G_0} V_0\cong \lambda_0^*(P\times _G V_0)$. That is, the pull-back of the core of the associated DVB is isomorphic to the associated bundle of the core of the DPB:
	\[
		\vcenter{\xymatrix{
			P\times_{G_0} V_0 \ar[d] \ar[r]^(0.45){\sim} &\lambda^*_0(P\times _G V_0) \ar[d] \ar[r] &P\times_G  V_0 \ar[d]\\
			P_1\times_M P_2\ar[r]^{\id} &P_1\times_M P_2 \ar[r]^(0.6){\lambda_0} &M
		}}.
	\]
\end{Pro}

\begin{proof}
	Here we will use $[p]$ and $\lfloor p,v\rfloor$ to denote the element in $P_1\times_M P_2=P/G_0$ and $P\times_{G_0}V_0$, where $p\in P$ and $v\in V_0$. By definition,
	\[
		\lambda_0^*(P\times _G V_0)=P/G_0\times_M(P\times _G V_0)=\{([p],[q,v])|\lambda_0([p])=\pi(q)\}.
	\]
	Since $\pi(p)=\lambda_0([p])=\pi(q)$, there exists a unique $g\in G$ such that $q=pg$, which means we can rewrite every element in $\lambda_0^*(P\times_G V_0)$ as $([p],[p,gv])$. Then define
	\[
		\Phi: \lambda_0^*(P\times _G V_0)\to P\times_{G_0} V_0:  ([p],[q,v])\mapsto \lfloor p,gv\rfloor.
	\]
	First it is well-defined: $\Phi([ph],[ph,h^{-1}gv_0])=\lfloor ph,h^{-1}gv_0\rfloor=\lfloor p,gv_0\rfloor=\Phi([p],[p,gv_0])$, where $h\in G_0$. Secondly we check it is a bundle map and show that $\Phi_{[p]}$ is linear. For $([p],[q,v])$ and $([p],[q',v'])\in \lambda_0^*(P\times _G V_0)$, there exist $g,g'\in G$ such that $q=pg$ and $q'=pg'$. Thus by definition,
	\begin{align*}
		\Phi_{[p]}(([p],[q,v])+([p],[q',v']))&=\lfloor p,gv+g'v'\rfloor,\\
		\Phi_{[p]}([p],[q,v])+\Phi_{[p]}([p],[q',v'])=\lfloor p,gv\rfloor+\lfloor p,g'v'\rfloor&=\lfloor p,gv+g'v'\rfloor.
	\end{align*}
	Then it is clear that $\Phi_{[p]}$ is an isomorphism of vector spaces, which ends the proof.
\end{proof}

To show that a double vector bundle can be always realized as an associated  bundle of a $\mathbb{G}$-DPB, we now define the frame bundle of a DVB $\mathbb{E}=(E;E_1,E_2;M)_{E_0}$. Suppose the fiber dimension of the vector bundle $E_i$ is $n_i$ and take $n=n_1+n_2+n_0$. Let $[\bm{n}]$ be the triple $(n;n_1,n_2)$, which is called the \emph{fiber dimension} of $\mathbb{E}$. Let $\mathbb{R}^{[\bm{n}]}$ be the trivial DVS of fiber dimension $[\bm{n}]$:
\[
	\mathbb{R}^{[\bm{n}]}:=(\mathrm{R}^{n_1}\oplus \mathrm{R}^{n_2}\oplus \mathrm{R}^{n_0};\mathrm{R}^{n_1},\mathrm{R}^{n_2})_{\mathrm{R}^{n_0}}.
\]
We first consider the frame bundle of a DVS $\mathbb{V}=(V;V_1,V_2)_{V_0}$. By \cite{notes}, given an injective morphism of DVSs $\psi:V_1\oplus V_2\to V$ over the identity on $V_i$, which is called a \emph{linear splitting} of $\mathbb{V}$, a corresponding decomposition  of $\mathbb{V}$ is given by:
\begin{equation}\label{Decom}
	\Psi:V_1\oplus V_2\oplus V_0\xrightarrow{\sim} V:(v_1,v_2,v_0)\mapsto \psi(v_1,v_2)+_2(0_{II}(v_2)+_1 v_0).
\end{equation}
Given a bilinear form $\mu\in T=\Hom(V_1\otimes V_2, V_0)$, we get another linear splitting of $\mathbb{V}$:
\begin{equation}\label{Action}
	\psi'(v_1,v_2):=(\mu(v_1,v_2)+_10_{II}(v_2))+_2\psi(v_1,v_2).
\end{equation}
Besides, any two splittings differ by a bilinear form $\mu\in T$, which means the space  of all decompositions, denoted by $D(\mathbb{V})$, is an affine space  by translation of the vector space $T$. The transformations between different decompositions constitute the group $\Aut(V_1\oplus V_2 \oplus V_0)$ in the following sense:
\begin{equation*}
	\vcenter{\xymatrix @C=2.7pc @R=.7pc{
		& V_1\oplus V_2 \oplus V_0 \ar[dd]^a \ar[dl]_(.6){\Psi'} \\
		V & \\
		& V_1\oplus V_2 \oplus V_0 \ar[ul]^(0.6){\Psi}
	}}	
	\qquad \text{~where~}a:=\Psi^{-1}\circ\Psi' \in \Aut(V_1\oplus V_2 \oplus V_0).
\end{equation*}
By Equation \eqref{Decom}, direct computation gives
\begin{equation}\label{CoorT}
	a=\Psi^{-1}\circ\Psi'=(\id,\id,\id,\mu)\in \Aut(V_1\oplus V_2 \oplus V_0).
\end{equation}
On the other hand, given bases $\mathbf{u},\mathbf{v}$ and $\mathbf{w}$ of $V_1$, $V_2$ and $V_0$, any point $v\in V$ is uniquely decided by a triple $\xi:=(x,y,z)\in \mathrm{R}^{n_1}\oplus \mathrm{R}^{n_2}\oplus \mathrm{R}^{n_0}$
\[
	v=\Psi(\mathbf{u}x,\mathbf{v}y,\mathbf{w}z),
\]
where
\[
	\mathbf{u}x:=\sum u_ix^{i} \in V_1, \quad	 \mathbf{u}=(u_1,\cdots,u_{n_1}) \mbox{~and~} x=(x^{1},\cdots,x^{n_1}).
\]
Thus we introduce the following definition:
\begin{Def}
	For a double vector space $\mathbb{V}=(V;V_1,V_2)_{V_0}$, a \emph{frame} $\Omega$ of $V$ is a quadruple
	\[
		\Omega:=(\mathbf{u},\mathbf{v},\mathbf{w};\Psi)\in \Pi^{n_1} V_1\times \Pi^{n_2} V_2\times \Pi^{n_0} V_0\times D(\mathbb{V})
	\]
	where $\mathbf{u},\mathbf{v}$, $\mathbf{w}$ are frames of $V_1$, $V_2$, $V_0$ respectively and $\Psi$ is a decomposition.  Denote by $\mathcal{F}(V)$ the set of frames of $V$.
\end{Def}

For the DVB $\mathbb{E}$, fix $m\in M$, then $\mathbb{E}_m:=(E_m;E_{1,m},E_{2,m})_{E_{0,m}}$ is a DVS. Let $\mathcal{F}(E)=\bigsqcup_m \mathcal{F}(E_m)$, which can be viewed as a submanifold of the following fiber product if we fix a decomposition of $\mathbb{E}$
\[
	\Pi^{n_1} E_1\times_M \Pi^{n_2} E_2\times_M \Pi^{n_0} E_0\times_M \Hom(E_1\otimes E_2, E_0).
\]
Given a frame $\Omega=(\mathbf{u},\mathbf{v},\mathbf{w};\Psi)\in \mathcal{F}(E_m)$ at $m$ and a triple $\xi=(x,y,z)\in \mathrm{R}^{n_1}\oplus \mathrm{R}^{n_2}\oplus \mathrm{R}^{n_0}$, let
\begin{equation}\label{Omegaxi}
    	\Omega\xi:=\Psi(\mathbf{u}x,\mathbf{v}y,\mathbf{w}z)\in E_m.
\end{equation}
 The next lemma  is easy to be checked by  the analysis above.
\begin{Lem}
	The following natural action of the group $\Aut(\mathrm{R}^{[\bm{n}]})$ on $\mathcal{F}(E)$ defined by:
	\[
		\Omega a:=(\mathbf{u}a_1,\mathbf{v}a_2,\mathbf{w}a_0;\Psi_{\mu}), \quad a=(a_1,a_2,a_0,\widetilde{\mu})\in  \Aut(\mathrm{R}^{[\bm{n}]})
	\]
	is free and transitive, where $\mu\in \Hom(E_{1,m}\otimes E_{2,m}, E_{0,m})$ is given by
	\begin{equation}\label{action}
		\mu(\mathbf{u}x,\mathbf{v}y):=\mathbf{w}\widetilde{\mu}(a^{-1}_1x,a^{-1}_2y)\in E_{0,m},
	\end{equation}
	and
	\[
		(\Psi_{\mu})(v_1,v_2,v_0):=(\mu(v_1,v_2)+_1 v_0+_10_{II}(v_2))+_2\Psi(v_1,v_2,0).
	\]
\end{Lem}

More detailed correspondence between 	$\widetilde{\mu}$ and $\mu$ is as follows. Let ${e_i}$, ${f_j}$ and ${l_k}$ be the natural bases of $\mathrm{R}^{n_1}$, $\mathrm{R}^{n_2}$ and $\mathrm{R}^{n_0}$ and ${e^i}$, ${f^j}$ be the dual bases. If we denote $\overline{\mathbf{u}}=\mathbf{u}a_1$, $\overline{\mathbf{v}}=\mathbf{v}a_2$ and $\overline{\mathbf{w}}=\mathbf{w}a_0$, Equation \eqref{action} is actually:
\[
	\widetilde{\mu}= \Sigma\mu_{ij}^k e^i\otimes f^j\otimes l_k\Rightarrow
	\mu=\Sigma \mu_{ij}^k \overline{u}^i\otimes \overline{v}^j\otimes \overline{w}_k,
\]
where $(\overline{u}^1,\cdots,\overline{u}^{n_1})$ is the dual basis of $\overline{\mathbf{u}}=(\overline{u}_1,\cdots,\overline{u}_{n_1})$. 

\begin{Thm}\label{FrameB3}
	Let $\mathbb{E}=(E;E_1,E_2)_M$ be a DVB with fiber dimension $[\bm{n}]$ and $\mathcal{F}(E_i)$ be the frame bundle of $E_i\to M, i=1,2$.

	(1) The quadruple $\mathcal{F}(\mathbb{E}):=(\mathcal{F}(E);\mathcal{F}(E_1),\mathcal{F}(E_2);M)$ is an $\Aut(\mathbb{R}^{[\bm{n}]})$-DPB, which is called the \emph{frame bundle} of $\mathbb{E}$. 
	
 	(2) The associated bundle $\mathcal{F}(\mathbb{E})\times_{\Aut(\mathbb{R}^{[\bm{n}]})}\mathbb{R}^{[\bm{n}]}$ is naturally isomorphic to the DVB $\mathbb{E}$.
\end{Thm}
\begin{Rm}
	By the above theorem, $\mathcal{F}(\mathbb{V})=(\mathcal{F}(V);\mathcal{F}(V_1),\mathcal{F}(V_2);\pt)$ turns out to be a DPB  over a  point if $\mathbb{E}=\mathbb{V}$ is a DVS. Similar to the case of a vector space, the frame bundle of a DVS is isomorphic to its structure group $\Aut(\mathbb{R}^{[\bm{n}]})$  as DPBs, but  there is no canonical  isomorphism  between them.
\end{Rm}
\begin{proof}[Proof of Theorem \ref{FrameB3}.]
	Note that $\Aut(\mathbb{R}^{[\bm{n}]})$ is the following DLG:
	\[
		\Aut(\mathbb{R}^{[\bm{n}]}): 1\to \GL(\mathrm{R}^{n_0})\ltimes T \to \Aut(\mathrm{R}^{[\bm{n}]}) \xrightarrow{\varphi} \GL(\mathrm{R}^{n_1})\times \GL(\mathrm{R}^{n_2}) \to 1,
	\]
	where $T=\Hom (\mathrm{R}^{n_1}\otimes \mathrm{R}^{n_2}, \mathrm{R}^{n_0})$.
	Let $K_1=\varphi^{-1}(\{e\}\times \GL(\mathrm{R}^{n_2}))$ and $K_2=\varphi^{-1}(\GL(\mathrm{R}^{n_1})\times\{e\})$.

	For (1), if $\pi:\mathcal{F}(E)\to M$ is a principal bundle, where $\pi$ is the natural projection, we will have $\mathcal{F}(E)/K_i\cong \mathcal{F}(E_i)$. Hence $\mathcal{F}(E)\to \mathcal{F}(E_i)$ is a principal $K_i$-bundle. To show $\pi:\mathcal{F}(E)\to M$ is a principal $\Aut(\mathrm{R}^{[\bm{n}]})$-bundle, we need the local trivialization, which is constructed as follows: Fix a decomposition $\Psi_0$. For a small enough open subset $U\subset M$, take a local section:
	\[
		s: U\subset M\to \mathcal{F}(E_1)\times_M \mathcal{F}(E_2)\times_M \mathcal{F}(E_0),
	\]
	thus we know $(s(m);\Psi_{0,m})\in \mathcal{F}(E)_m$. Let
	\begin{align*}
		\alpha: \pi^{-1}(U)&\to U\times \Aut(\mathrm{R}^{[\bm{n}]})\\
		\Omega&\mapsto (\pi(\Omega), a), \mbox{~where~} a \mbox{~is uniquely defined by~} \Omega=(s(\pi(\Omega));\Psi_{0,\pi(\Omega)})a,
	\end{align*}
	which is well-defined and 1-1. Then we need to show that $\alpha$ is $\Aut(\mathrm{R}^{[\bm{n}]})$-equivariant. Let $\pi(\Omega)=m$ and suppose $\alpha(\Omega b)=(m,a')$. Then by definition, $\Omega b=(s(m);\Psi_{0,m})a'$, thus $\Omega=(s(m);\Psi_{0,m})a'b^{-1}$, which gives $a'b^{-1}=a$. Hence $\alpha(\Omega b)=(m,a')=(m,ab)=\alpha(\Omega)b$.

	To see (2),  we define:
	\[
		f: \mathcal{F}(E)\times_{\Aut(\mathrm{R}^{[\bm{n}]})}\mathrm{R}^{[\bm{n}]}\to E: [\Omega,\xi]\mapsto \Omega\xi,
	\]
	where  $\Omega\xi:=\Psi(\mathbf{u}x,\mathbf{v}y,\mathbf{w}z)\in E$ is defined in \eqref{Omegaxi}.
	To show that $f$ is well-defined, we need to prove
	\begin{equation}\label{well-defined}
		\Omega\xi=(\Omega a)\cdot(a^{-1}\xi),\qquad \forall a=(a_1,a_2,a_0,\widetilde{\mu})\in \Aut(\mathrm{R}^{[\bm{n}]}).
	\end{equation}
	In fact, first note that $a^{-1}=(a^{-1}_1,a^{-1}_2,a^{-1}_0,-a^{-1}_0\circ\widetilde{\mu}\circ(a^{-1}_1\times a^{-1}_2))$. Then we get
	\[
		a^{-1}\xi=a^{-1}(x,y,z)=(a^{-1}_1x,a^{-1}_2y,a^{-1}_0z-a^{-1}_0\widetilde{\mu}(a^{-1}_1x, a^{-1}_2y)).
	\]
	On the other hand, since $\Omega a=(\mathbf{u}a_1,\mathbf{v}a_2,\mathbf{w}a_0; \Psi_{\mu})$, by Equation \eqref{Omegaxi}, it follows that
	\[
		(\Omega a)\cdot(a^{-1}\xi)=\Psi_{\mu}(\mathbf{u}x,\mathbf{v}y,\mathbf{w}z-\mathbf{w}\tilde{\mu}(a^{-1}_1x, a^{-1}_2y)).
	\]
	By Equation \eqref{CoorT}, we know
	\[
		\Psi_{\mu}^{-1}\circ\Psi=(\id,\id,\id,-\mu),
	\]
	which yields Equation \eqref{well-defined}. It is routine to check that $f$ is 1-1 and hence induces an isomorphism of vector bundles.
\end{proof}

Our next concern is to look at the linear sections and core sections of $P\times_G V$ on $P_2\times_{G_2} V_2$ in the DVB $\mathbb{P}\times_{\mathbb{G}} \mathbb{V}$. For a DVB $\mathbb{E}=(E;E_1,E_2;M)$, first let us recall the definition of these two special types of sections of the vector bundle $E_{II} : E\rightarrow E_2$: 
\begin{Def}\emph{\cite{gracia2010lie}}
	A section $s\in \Gamma(E_{II})$ is \emph{linear} if it is a bundle morphism from $E_2$ to $E_I$. The space of linear sections is denoted as $\Gamma_{\ell}(E_{II})$. A section $s\in \Gamma(E_{II})$ is called a \emph{core section} if it is the image of the following embedding:
	\[
		t\in \Gamma(E_0)\hookrightarrow\bar{t}:=t\circ\delta_2+_10_{II}\in \Gamma(E_{II}),
	\]
	where $\delta_2: E_2\to M$ is the projection and $0_{II}: E_2\rightarrow E$ is the zero section. Denote by $\Gamma_C(E_{II})$ the space of core sections.
\end{Def}

The following proposition characterizes the general sections of $P\times_G V$ on $P_2\times_{G_2} V_2$:

\begin{Pro} The section space 
	$\Gamma(P\times_G V,P_2\times_{G_2} V_2)$ is  isomorphic to the function space:
	\[
		C(P\times V_2,V) = \{f:P\times V_2\to V\,|\, \sigma_2f(p,v_2)=v_2, \, f(pg,\varphi_2(g)^{-1} v_2)=g^{-1} f(p,v_2)\},
	\]
	where $p\in P$, $v_2\in V_2$ and $g\in G$. Moreover, one has 
	\[
		\Gamma(P\times_G V_0)\cong C(P,V_0):=\{f:P\rightarrow V_0\,|\, f(pg)=g^{-1} f(p)\}.
	\]
\end{Pro}
\begin{proof}
	For any $f\in C(P\times V_2,V)$, define a map $s_f$ by
	\[
		s_f: P_2\times_{G_2} V_2\to P\times_G V: [p_2,v_2]\mapsto[p,f(p,v_2)],\quad \forall p\in P\mbox{,~s.t.~} \pi_2(p)=p_2.
	\]
	The map $s_f$ is well-defined since it does not depend on the choice of $p$. Actually, for $p'=pk_2$, where $k_2\in K_2$, by the $g$-invariance of $f$, we get $$f(pk_2,v_2)=k_2^{-1}f(p,v_2) \Longrightarrow [pk_2,f(pk_2,v_2)]=[p,f(p,v_2)].
	$$
	Moreover, $s_f$ is a section as  $\sigma_2f(p,v_2)=v_2$. On the other hand, given a section $s\in \Gamma(P\times_G V,P_2\times_{G_2} V_2)$, define a function $f_s:P\times V_2\rightarrow V$ by the following relation:
	\[
		[p,f_s(p,v_2)]=s([\pi_2(p),v_2]).
	\]
	It is simple to see that $\sigma_2f_s(p,v_2)=v_2$. Moreover, by definition,
	\[
		[pg,f_s(pg,\varphi_2^{-1}(g)v_2)]=s([\pi_2(p)\varphi_2(g),\varphi_2^{-1}(g)v_2])=s([\pi_2(p),v_2])=[p,f_s(p,v_2)],
	\]
	which implies that $f_s(pg,\varphi_2^{-1}(g)v_2)=g^{-1}f_s(p,v_2)$. Thus we have $f_s\in C(P\times V_2,V)$. 

	We see right now that the previous two processes are invertible to each other. This ends the proof.
\end{proof}

Now it is clear that the space $\Gamma_{\ell}(P\times_G V,P_2\times_{G_2} V_2)$ of \emph{linear} sections is consist of functions $f:P\times V_2\to V$ such that  $f_p:V_2\to V \in \Gamma_{\ell}(V,V_2),\forall p\in P$. And the space $\Gamma_{\ell}(V,V_2)$ fits into the following exact sequence of vector spaces:
\begin{equation*}
	0\to V_2^*\otimes V_0\to \Gamma_{\ell}(V,V_2) \xrightarrow{} V_1 \to 0.
\end{equation*}
As the sections of $P\times_G V_0$ can be embedded into $\Gamma(P\times_G V,P_2\times_{G_2} V_2)$ by
\begin{eqnarray}\label{core}
c\in C(P,V_0)\mapsto \bar{c}\in  C(P\times V_2,V): (p,v_2)\mapsto c(p)+_1 0_{II}(v_2),
\end{eqnarray} the space $\Gamma_C(P\times_G V,P_2\times_{G_2} V_2)$ of \emph{core} sections is the image of \eqref{core}. 

At the end of this section, we  explain the duality of a DVB from the  viewpoint of  associated bundles. First let us recall what is the \emph{dual} of a double vector bundle (see \cite{konieczna1999double, mackenzie2005duality} for more details): Dualizing either structure on $E$ leads to a double vector bundle:
	\begin{equation*}	
       \vcenter{\xymatrix  @C=.6pc @R=.6pc{  
		E \ar[dd] \ar[rr] & &E_2 \ar[dd]\\
		 & E_0 \ar[dr]&\\
		E_1 \ar[rr] & &M
	}} \Longrightarrow
       \vcenter{\xymatrix  @C=.6pc @R=.6pc{ 
		E^*_I \ar[dd] \ar[rr] & &E^*_0 \ar[dd]\\
		 & E^*_2 \ar[dr]&\\
		E_1 \ar[rr] & &M
	}},
   \end{equation*}
where $E^*_I\to E_1$ is the dual vector bundle of $E_I$. The duality interchanges the positions of $E_2$ and $E_0$, which is natural from Theorem \ref{thm Aut} and Proposition \ref{Dual}: The dual of a bilinear map with respect to one component exchanges the positions of the other two spaces.

Remember that $\Aut(\mathbb{V})$ and $\Aut(\mathbb{V}_I^*)$ are automorphism groups of a DVS $\mathbb{V}=(V;V_1,V_2)_{V_0}$ and its dual $\mathbb{V}_I^*=(V_I^*;V_1,V_0^*)_{V_2^*}$ respectively as given in \eqref{automorphism group}. For an $\Aut(\mathbb{V})$-DPB  $\mathbb{P}=(P;P_1,P_2,M)$ and its dual $\Aut(\mathbb{V}_I^*)$-DPB $\mathbb{P}_1^*=(P;P_1,Q_2;M)$ as in Corollary \ref{Dual DPB}, we get two associated DVBs by Theorem \ref{main}:
\[
 	\mathbb{P}\times_{\Aut(\mathbb{V})} \mathbb{V}: \quad
	\vcenter{\xymatrix {
	P\times_{\Aut(V)} V \ar[d]^{\Pi_1} \ar[r] &P_2\times_{\GL(V_2)} V_2 \ar[d]\\
		P_1\times_{\GL(V_1)} V_1 \ar[r] &M
	}},
\]
and
\[
	\mathbb{P}_I^*\times_{\Aut(\mathbb{V}_I^*)} \mathbb{V}_I^*:\quad
	\vcenter{\xymatrix {
		P\times_{\Aut(V_I^*)} V_I^* \ar[d]^{\Pi_1^*} \ar[r] &Q_2\times_{\GL^*(V_0)} V_0^* \ar[d]\\
		P_1\times_{\GL(V_1)} V_1 \ar[r] &M
	}}.
\]
\begin{Pro}
	The two associated DVBs $\mathbb{P}\times_{\Aut(\mathbb{V})} \mathbb{V}$ and $\mathbb{P}_I^*\times_{\Aut(\mathbb{V}_I^*)} \mathbb{V}_I^*$ are dual to each other with respect to $P_1\times_{\GL(V_1)} V_1$.
\end{Pro}
\begin{proof}
	First we show that $P\times_{\Aut(V)} V$ and $P\times_{\Aut(V_I^*)} V_I^*$ are dual to each other as vector bundles over $P_1\times_{\GL(V_1)} V_1$. Take $[p,v]\in P\times_{\Aut(V)} V$ and $[q,\eta]\in P\times_{\Aut(V_I^*)} V_I^*$ such that $\Pi_1([p,v])=\Pi_1^*([q,\eta])$. Then there exists $g\in \Aut(V)$ such that $p=qg$. Thus the pairing can be defined by $\langle [p,v],[q,\eta]\rangle=\langle gv,\eta\rangle$. It is routine to check that the pairing is well-defined and non-degenerate. Then we check the core bundle $P\times_{\Aut(V)} V_0$ and $Q_2\times_{\GL^*(V_0)} V_0^*$ are dual to each other as vector bundles over  $M$. Take $[p,v]\in P\times_{\Aut(V)} V_0$ and $[[q],\eta]\in Q_2\times_{\GL^*(V_0)} V_0^*$ such that $\pi(p)=\pi(q)$, thus there exists $g\in \Aut(V)$ such that $p=qg$. Then the pairing can be defined by $\langle [p,v],[[q],\xi]\rangle=\langle gv,\xi\rangle$, which is also well-defined and non-degenerate.
\end{proof}

\begin{Rm}\label{Rm:4.12}
	In \cite{voronov, voronov2}, Voronov describes double vector bundles in the framework of graded manifolds. They are bi-graded manifolds, where the weights of coordinates are restricted to 0 and 1 only. The structure group, which corresponds to DLGs in our definition, consists of transformations respecting this bi-grading. Hence we believe that by considering more general bi-gradings, one can produce more examples of double Lie groups. Similarly, one could  generalize this construction to {\em multiple vector bundles} as well.

	Note that DVB in the original paper \cite{Pra} is described by using local trivialization and transition functions. This has been generalized to multiple vector bundles in \cite[Section 6.2]{voronov2012structure}. The structure group of such bundle structure hints the existence of the concepts of DLGs and the frame bundles of DVBs. Moreover, the constructions in \cite{voronov2012structure} provide a way to introduce {\em multiple Lie groups} and {\em multiple principal bundles}.
\end{Rm}

\section{Connections and Gauge Transformations}\label{App}

We first recall the definition of connections in a principal bundle. Let $P$ be a principal $G$-bundle over $M$ and 	$\widehat{\mathfrak{g}}$  denotes the vertical distribution on $P$ spanned by    the fundamental vector fields. That is
\[
 	\widehat{\mathfrak{g}}=\Span\{\hat{A}\,|\, \forall A\in \mathfrak{g}\}\subset TP, \quad
\hat{A}_p=\frac{d}{dt}\Big|_{t=0}p e^{tA}, \qquad\forall p\in P.
\]
 A \emph{connection} $H$ in $P$ is a $G$-invariant distribution on $P$ such that $TP=H\oplus \widehat{\mathfrak{g}}$.    
Here we propose a notion of connections in a DPB $\mathbb{P}=(P;P_1,P_2;M)$:
\[
	\vcenter{\xymatrix {
		P \ar[d]^{\pi_1} \ar[r]^{\pi_2} &P_2 \ar[d]^{\lambda_2}\\
		P_1 \ar[r]^{\lambda_1} &M
	}}.
\]
Set $\pi=\lambda_1\circ \pi_1=\lambda_2\circ \pi_2:P\rightarrow M$. Denote by $\pi^*(TM)$ the pullback bundle of $TM$ along $\pi$. 
\begin{Def}\label{projection connection}
	Let $\mathbb{P}=(P;P_1,P_2;M)$ be a DPB. A \emph{connection} in $\mathbb{P}$ is a pair of connections $(H_1,H_2)$ with $H_i$ in $P(K_i,P_i)$ for $i=1,2$ such that
	\begin{itemize}
		\item[1)] the intersection  $H_1\cap H_2$ is a $K_i$-invariant subbundle in $TP$ (i=1,2);
		\item[2)] the projection $H_1\cap H_2\rightarrow \pi^*(TM)$ is a bundle isomorphism. Namely, for any $p\in P$, the map $\pi_{*,p}: (H_1 \cap H_2)_p\to T_{\pi(p)} M$  is an isomorphism of vector spaces.
	\end{itemize}
	A connection $(H_1,H_2)$ in $\mathbb{P}$ is called flat if $H_i$ is a \emph{flat} connection in $P(K_i,P_i)$ for $i=1,2$. 
\end{Def}
\begin{Ex}
	Let $\mathbb{G}=(G;G_1,G_2)$ be a double Lie group, seen as a DPB. Let $H_1$ be a connection in $G(K_1,G_1)$ and $H_2$ a connection in $G(K_2,G_2)$. Then $(H_1,H_2)$ is a connection in $\mathbb{G}$ iff $H_1\cap H_2=\{0\}$.
\end{Ex}
Note that if $(H_1,H_2)$ is a connection in a DPB $\mathbb{P}$, then we have $\dim(H_1+H_2)=\dim(P_1)+\dim(P_2)-\dim(M)$. So the distribution $H_1+H_2$ becomes a connection in the core $P(G_0,P_1\times_M P_2)$ iff  it is transversal to the vertical distribution $\widehat{\mathfrak{g}_0}$.
The following theorem gives the motivation for the definition of connections in a DPB. 
\begin{Thm}\label{Push-forward a pair}
	Let $(H_1,H_2)$ be a connection in a DPB $\mathbb{P}=(P;P_1,P_2;M)$. Then
	\begin{itemize}
		\item[(1)] $L_i=\pi_{i*}(H_1\cap H_2)$ gives a connection in $P_i$ for $i=1,2$. Moreover, the connections $L_1$ and $L_2$ are flat if $(H_1,H_2)$ is  flat. 
		\item[(2)] the distribution $H_1+H_2$ is a connection in the core $P(G_0,P_1\times_M P_2)$ if and only if $(H_1+H_2)\cap \widehat{\mathfrak{g}_0}=\{0\}$.
	 \end{itemize}
\end{Thm} 
\begin{proof}
  For any $p_1\in P_1$, set $L_{1,p_1}:=\pi_{1*,p}(H_1\cap H_2)_p$ with $p$ such that $\pi_1(p)=p_1$. Define $L_2$ similarly. We claim $L_1,L_2$ are connections in $P_1$ and $P_2$ iff $(H_1,H_2)$ is a connection in the DPB $\mathbb{P}$. 
	First, we show $L_{1,p_1}$ is well-defined iff $H_1\cap H_2$ is $K_1$-invariant. Indeed, for $p'\in P$ such that $\pi_1(p')=p_1$, we have $p'=pk_1$ for some $k_1\in K_1$. Then following from $\pi_1=\pi_1\circ R_{k_1}$, we have that $L_{1,p_1}$ being well-defined means
	\[
		\pi_{1*,p}(H_1\cap H_2)_p =\pi_{1*,pk_1}(H_1\cap H_2)_{pk_1}=\pi_{1*,pk_1}R_{k_1*}(H_1\cap H_2)_p,
	\]
	which implies exactly
	\[
		R_{k_1*}(H_1\cap H_2)_p=(H_1\cap H_2)_{pk_1}.
	\]
	This is due to the fact that $R_{k_1*}(H_1\cap H_2)_p, (H_1\cap H_2)_{pk_1}\subset H_{1,pk_1}$ and $\pi_{1*,pk_1}:H_{1,pk_1}\to T_{p_1} P_1$ is an isomorphism. Moreover, we can choose $p$ such that $L_{1,p_1}$ depends on $p_1$ smoothly. Similarly, $L_{2,p_2}$ is well-defined iff $H_1\cap H_2$ is $K_2$-invariant.

	We now turn to check that $L_{1,p_1}$ is $G_1$-invariant, i.e. $R_{g_1*}L_{1,p_1}=L_{1,p_1g_1}.$ Choosing $k_2\in K_2$ such that $\phi_1(k_2)=g_1$, we have $\pi_1(pk_2)=p_1g_1$. Since $R_{g_1}\circ \pi_1=\pi_1\circ R_{k_2}$, we have
	\[
		R_{g_1*}L_{1,p_1}=R_{g_1*}\pi_{1*,p}(H_1\cap H_2)_p=\pi_{1*,pk_2}R_{k_2*}(H_1\cap H_2)_p.
	\]
	Since $H_1\cap H_2$ is $K_2$-invariant,  we obtain $R_{g_1*}L_{1,p_1}=L_{1,p_1g_1}$. It remains to prove
	\[
		T_{p_1} P_1=L_{1,p_1}\oplus \widehat{\mathfrak{g}_1}_{,p_1}.
	\]
	As $\lambda_{1*,p_1}(\widehat{\mathfrak{g}_1}_{,p_1})=0$,
	this holds iff $\lambda_{1*,p_1}:L_{1,p_1}\to T_{\pi_1(p_1)} M$ is an isomorphism. This is further equivalent to the fact that $\pi_{*,p}: (H_1\cap H_2)_p\to T_{\pi(p)} M$ is an isomorphism since $\pi_{1*,p}: (H_1\cap H_2)_p\to L_{1,p_1}$ is an isomorphism. This ends the proof.
\end{proof}

In particular, when $\mathbb{P}$ is a $\mathbb{G}$-DPB, given a  connection in $\mathbb{P}$, there exists a natural connection in the principal $G$-bundle $P(G,M)$. Conversely, under some assumptions, we can get a connection in $\mathbb{P}$ from a connection in $P(G,M)$. Moreover, there is also a connection in the core $P(G_0,P_1\times _M P_2)$.

\begin{Pro}
	Suppose that $(H_1,H_2)$ is a connection in a $\mathbb{G}$-DPB $\mathbb{P}$. Then the distribution $H_1\cap H_2$ is a connection in the principal $G$-bundle $P(G,M)$, which is flat if $(H_1,H_2)$ is flat.
\ \end{Pro}
\begin{proof}
	Since $H_1\cap H_2$ is $K_1$ and $K_2$-invariant, it is $G=K_1K_2$-invariant. Note that $\pi_{*,p}: (H_1 \cap H_2)_p\to T_{\pi(p)} M$ is an isomorphism and $\pi_{*,p}(\widehat{\mathfrak{g}})=0$. We have $(H_1\cap H_2)_p\cap \widehat{\mathfrak{g}}_p=0$. Comparing the dimensions, we get\[
	 (H_1\cap H_2)_p\oplus \widehat{\mathfrak{g}}_p=T_p P.\]
	So $H_1\cap H_2$ is a connection in the principal $G$-bundle $P(G,M)$.
\ \end{proof}

Let $\Lie(\mathbb{G})=(\mathfrak{g};\mathfrak{g}_1,\mathfrak{g}_2)_{\mathfrak{g}_0}$ be the Lie algebra of the DLG $\mathbb{G}$. We have an exact sequence of Lie algebras:
	\[
		\xymatrix@C=-1.3em{
		*+[l]{0 \to \mathfrak{g}_0\to \mathfrak{g}}\ar[r] & *+[r]{\mathfrak{g}_1 \oplus \mathfrak{g}_2 \to 0}\ar@/^/[l]
		}
	\]
Given a splitting $s:\mathfrak{g}_1\oplus \mathfrak{g}_2\rightarrow \mathfrak{g}$, the vector space $\mathfrak{g}$ is  decomposed as $\mathfrak{g}=\mathfrak{g}_0\oplus s(\mathfrak{g}_1\oplus \mathfrak{g}_2)$.
\begin{Pro}
	Let $H$ be a connection in the principal $G$-bundle $P(G,M)$. Extend it to two distributions 
	\[H_i:=H\oplus \widehat{s(\mathfrak{g}_i)}\subset TP,\qquad i=1,2.\]
	Obviously, we have $H=H_1\cap H_2$.
	Then $(H_1,H_2)$
	gives a connection in the $\mathbb{G}$-DPB $\mathbb{P}$ iff
	\[
		[\mathfrak{g}_0,s(\mathfrak{g}_1\oplus \mathfrak{g}_2)]=0,\quad \quad [s(\mathfrak{g}_1),s(\mathfrak{g}_2)]=0.
	\]
	Moreover, the extension $H\oplus \widehat{s(\mathfrak{g}_1\oplus \mathfrak{g}_2)}\subset TP$ is a connection in the core $P(G_0,P_1\times_M P_2)$ iff
	\[
		[\mathfrak{g}_0,s(\mathfrak{g}_1\oplus \mathfrak{g}_2)]=0.
	\] 
\end{Pro}
\begin{proof}
	To see that $H_1$ is a connection in $P(K_1,P_1)$, first note that $\mathfrak{k_1}=\mathfrak{g}_0\oplus s(\mathfrak{g}_2)$. Following from $TP=H\oplus \hat{\mathfrak{g}}$, we have $TP=H_1\oplus \widehat{\mathfrak{k_1}}$. Next, $H_1$ being $K_1$-invariant requires that $[\mathfrak{k}_1,s(\mathfrak{g}_1)]\subset s(\mathfrak{g}_1)$. Since $K_1,K_2$ are normal groups of $G$, we have $[\mathfrak{k}_1,\mathfrak{k}_2]\subset \mathfrak{g}_0$. Also, observe that $s(\mathfrak{g}_1)\subset \mathfrak{k}_2$ and  $s(\mathfrak{g}_1)\cap \mathfrak{g}_0=\{0\}$.  The requirement $[\mathfrak{k}_1,s(\mathfrak{g}_1)]\subset s(\mathfrak{g}_1)$ becomes $[\mathfrak{k}_1,s(\mathfrak{g}_1)]=0$, which is equivalent to the conditions $[\mathfrak{g}_0,s(\mathfrak{g}_1)]=0$ and $[s(\mathfrak{g}_2),s(\mathfrak{g}_1)]=0$. The conditions for $H_2$ being a connection in $P(K_2,P_2)$ is similar to get. 

	By the fact that $TP=H\oplus \hat{\mathfrak{g}}$, we have $TP=(H\oplus \widehat{s(\mathfrak{g}_1\oplus \mathfrak{g}_2)})\oplus \widehat{\mathfrak{g}_0}$. Then, $H\oplus \widehat{s(\mathfrak{g}_1\oplus \mathfrak{g}_2)}$ is $G_0$-invariant iff $\widehat{s(\mathfrak{g}_1\oplus \mathfrak{g}_2)}$ is $G_0$-invariant, iff $[\mathfrak{g}_0,s(\mathfrak{g}_1\oplus \mathfrak{g}_2)]=0$.
\end{proof} 
\begin{Ex}
	Suppose that the total space $P$ of a $\mathbb{G}$-DPB $\mathbb{P}=(P;P_1,P_2;M)$ is equipped with a $G$-invariant Riemannian metric. Define a distribution $H_i:=\widehat{\mathfrak{k}_i}^\perp$ $(\mathfrak{k}_i=\Lie(K_i))$, that is, $TP=H_i\oplus \mathfrak{k}_i$. Then $H_i$ is naturally a connection in $P(K_i,P_i)$ $(i=1,2)$. Consider 
    the intersection \[H_1\cap H_2=\widehat{\mathfrak{k}_1}^\perp\cap \widehat{\mathfrak{k}_2}^\perp=(\widehat{\mathfrak{k}_1}+\widehat{\mathfrak{k}_2})^\perp=\hat{\mathfrak{g}}^\perp.\]
    We find it
    is a connection in the principal $G$-bundle $P(G,M)$ and $(H_1,H_2)$ is a connection in the DPB $\mathbb{P}$. We also get a connection $\pi_{i,*}(H_1\cap H_2)$ in the principal $G_i$-bundle $P_i (i=1,2)$. Moreover, the summation
	\[H_1+H_2=\widehat{\mathfrak{k}_1}^\perp+\widehat{\mathfrak{k}_2}^\perp=(\widehat{\mathfrak{k}_1}\cap \widehat{\mathfrak{k}_2})^\perp=\widehat{\mathfrak{g}_0}^\perp\] is a connection in the core $P(\widehat{G}_0,P_1\times _M P_2)$.
\end{Ex} 

In Theorem \ref{Push-forward a pair}, we push forward a pair of connections in $P(K_i,P_i)$ for $i=1,2$ to obtain connections in $P_1$ and $P_2$. Now we move to consider the push-forward of one connection in $P(K_2,P_2)$ to achieve a connection in $P_1$.
\begin{Pro}
	Let $\mathbb{P}=(P;P_1,P_2;M)$ be a DPB. The push-forward of a connection $H_2$ in $P(K_2,P_2)$ by $\pi_1$ gives a connection in $P_1$ if and only if $H_2$ satisfies
	\[
		R_{k_1 *} H_{2}+\widehat{\mathfrak{k}_1} \subset H_2\oplus \widehat{\mathfrak{g}_0},\quad \quad \forall k_1\in K_1.
	\]
\end{Pro}  
\begin{proof}
   For any $p_1\in P_1$, define $L_{1,p_1}=\pi_{1*,p}H_{2,p}$ with $p$ such that $\pi_1(p)=p_1$. We check $L_1$ defines a connection in $P_1$ iff $H_2$ satisfies the above condition.
	First, note that $\ker \pi_{1*,p}=\widehat{\mathfrak{k}_1}_{,p_1}$. It is easily seen that  $L_{1,p_1}$ does not depend on the choice of $p$ if and only if $H_2$ satisfies
	\[
		R_{k_1 *} H_{2,p}\subset H_{2,pk_1}\oplus \widehat{\mathfrak{k}_1}_{,pk_1}.
	\]
	Then for $g_1\in G_1$, choose $k_2\in K_2$ such that $\phi_1(k_2)=g_1.$ By the fact that $\pi_1(pk_2)=\pi_1(p)g_1$, we find
	\[
		R_{g_1 *}L_{1,p_1}=R_{g_1 *}\pi_{1*,p}H_{2,p}=\pi_{1*,pk_2}R_{k_2*}H_{2,p}=\pi_{1*,pk_2}H_{2,pk_2}=L_{1,p_1g_1}.
	\]
	Namely, $L_1$ is $G_1$-invariant. Since $\pi_1$ is surjective and $\pi_{1*}(\widehat{\mathfrak{k}}_2)=\widehat{\mathfrak{g}}_1$, we have
	\[
		L_{1,p_1}+\widehat{\mathfrak{g}_1}_{,p_1}=T_{p_1} P_1.
	\]
	Then we claim that $L_{1,p_1}\cap \widehat{\mathfrak{g}_1}_{,p_1}=\{0\}$ if and only if $\widehat{\mathfrak{k}_1}_{,p}\subset H_{2,p}\oplus \widehat{\mathfrak{g}_0}_{,p}$. Actually, if the latter holds, assuming $\pi_{1*}(X)=\pi_{1*}(\hat{A})$ for $X\in H_{2,p}$ and $A\in \mathfrak{k}_2$, there exists $B\in\mathfrak{k}_1$ such that $X-\hat{A}=\hat{B}=B_2+\widehat{B_0}$, where $B_2\in H_{2,p}$ and $B_0\in \mathfrak{g}_0$. The condition $H_2\cap \widehat{\mathfrak{k}_2}=\{0\}$ implies that $X-B_2=\hat{A}+\widehat{B_0}=0$, so we get $A\in \mathfrak{g}_0$. Namely, $\pi_{1*}(X)=\pi_{1*}(\hat{A})=0$. For the `only if' direction, suppose $A\in \mathfrak{k}_1$ such that $\hat{A}_p=Y+\widehat{C}\in H_{2,p}\oplus \widehat{\mathfrak{k}_2}_{,p}$ and $\phi_1(C)\neq 0$. Then we have $\pi_{1*,p}(Y)=-\widehat{\phi_1(C)}\neq 0$, which implies $L_{1,p_1}\cap \widehat{\mathfrak{g}_1}_{,p_1}\neq\{0\}$, a contradiction. This ends the proof.
\end{proof}

An \emph{automorphism} of a principal bundle $P(G,M)$ is a diffeomorphism $f:P\to P$ such that $f(pg)=f(p)g$. Note that $f$ induces a well-defined diffeomorphism $\bar{f}: M\to M$ given by $\bar{f}(\pi(p))=\pi(f(p))$. An automorphism $f$ of a principal bundle is called a \emph{gauge transformation} if $\bar{f}=\id_M$. Set $\Aut(P)$ and $\GA(P)$ the groups of automorphisms and gauge transformations.
\begin{Def}
	An \emph{automorphism} of a DPB $\mathbb{P}=(P;P_1,P_2;M)$ is a diffeomorphism $f:P\to P$ such that $f\in \Aut(P(K_1,P_1))\cap \Aut(P(K_2,P_2))$. An automorphism of a DPB is called a \emph{gauge transformation} if  $\pi(p)=\pi(f(p))$. 
\end{Def}
Denote the sets of automorphisms and gauge transformations also by $\Aut(P)$ and $\GA(P)$ respectively.

Automorphisms and gauge transformations of a $\mathbb{G}$-DPB $\mathbb{P}$ are exactly transformations and gauge transformations of the principal $G$-bundle $P(G,M)$.
\begin{Lem}
	If $f\in \GA(P)$, then the induced map $f_1:P_1\to P_1$ defined by $
		f_1(p_1)=\pi_1(f(p))$ for any $p\in P$ such that $\pi_1(p)=p_1$
	is a gauge transformation of $P_1$. Likewise, we can get a gauge transformation of $P_2$.
\end{Lem}
\begin{proof}
	 We need to check $f_1(p_1g_1)=f_1(p_1)g_1$ and $\lambda_1(p_1)=\lambda_1(f_1(p_1))$. Choosing an element $k_2\in K_2$ such that $\phi_1(k_2)=g_1$, we have
	 \[
	 	f_1(p_1g_1)=\pi_1(f(pk_2))=\pi_1(f(p)k_2)=\pi_1(f(p))\phi_1(k_2)=f_1(p_1)g_1.
	 \]
	 Moreover, we find
	\[
		\lambda_1(f_1(p_1))=\lambda_1(\pi_1(f(p)))=\pi(f(p))=\pi(p)=\lambda_1(p_1).
	\]
	Thus, we get $f_1$ is a gauge transformation of $P_1$.
\end{proof}
\begin{Pro}
	We have the infinite dimensional ``double Lie groups'', $\Aut(\mathbb{P})$ and    $\GA(\mathbb{P})$, as follows:
   \[
		\vcenter{\xymatrix{
			\Aut(P) \ar[d] \ar[r] &\Aut(P_2) \ar[d]\\
			\Aut(P_1) \ar[r] &\Diff(M)
		}},\qquad
		\vcenter{\xymatrix{
			\GA(P) \ar[d]^{\varphi_1} \ar[r]^{\varphi_2} &\GA(P_2) \ar[d]\\
			\GA(P_1) \ar[r] &\id_M
		}}.
	\]
\end{Pro}
The ``double Lie algebra'' of the automorphism ``double Lie group'' of a DPB $\mathbb{P}$ is
\[
	\vcenter{\xymatrix{
		\mathfrak{X}(P)^{K_1}\cap \mathfrak{X}(P)^{K_2}\ar[d] \ar[r]&\mathfrak{X}(P_2)^{G_2} \ar[d]\\
		\mathfrak{X}(P_1)^{G_1} \ar[r] &\mathfrak{X}(M)
	}}.
\]
Here $\mathfrak{X}(P)^{K_1}$ denotes the space of $K_1$-invariant vector fields on $P$, which is the section of the gauge algebroid $TP/K_1$ of the principal $K_1$-bundle $P(K_1,P_1)$.  The space $\mathfrak{X}(P_1)^{G_1}$, consisting of $G_1$-invariant vector field on $P_1$, is in fact the section of the gauge algebroid $TP_1/G_1$ of $P_1$.  Actually, the space $\mathfrak{X}(P)^{K_1}\cap \mathfrak{X}(P)^{K_2}$ is the  pre-image of $\mathfrak{X}(P_1)^{G_1}\times \mathfrak{X}(P_2)^{G_2}$ under the map $(\pi_{1*},\pi_{2*}):TP\to TP_1\times TP_2$.

In general, for a DPB $\mathbb{P}=(P;P_1,P_2,M)$, given a vector field on $M$, by two steps of horizontal lifts, we get a $K_1$-invariant vector field and a $K_2$-invariant vector field on $P$. But they may not coincide. With the existence of a connection in $\mathbb{P}$, we have the following proposition.
\begin{Pro}
	The map $\pi_*: \mathfrak{X}(P)^{K_1}\cap \mathfrak{X}(P)^{K_2}\rightarrow \mathfrak{X}(M)$ is surjective if there is a connection $(H_1,H_2)$ in a DPB $\mathbb{P}=(P;P_1,P_2,M)$.
\end{Pro}
\begin{proof}
	This is a direct consequence of $H_1\cap H_2\cong \pi^*(TM)$ and $ \Gamma(H_1\cap H_2)\subset \mathfrak{X}(P)^{K_1}\cap \mathfrak{X}(P)^{K_2}$.
\end{proof}
\begin{Pro}
	Let $f\in \Aut(\mathbb{P})$ be a gauge transformation of a DPB $\mathbb{P}$ and $(H_1,H_2)$ a connection in $\mathbb{P}$. Then $(f_*H_1,f_*H_2)$ is also a connection in $\mathbb{P}$.
\end{Pro}
\begin{proof}
	First, since $f\in \Aut(P(K_1,P_1))\cap \Aut(P(K_2,P_2))$, the distributions $f_* {H_1}$ and $f_* {H_2}$ are connections in $P(K_1,P_1)$ and $P(K_2,P_2)$ respectively. Then, since $f(pk_1)=f(p)k_1$ and $f(pk_2)=f(p)k_2$, we get that $f_* {H_1} \cap f_* {H_2}=f_* (H_1\cap H_2)$ is $K_1$ and $K_2$-invariant on the condition that $H_1\cap H_2$ is $K_1$ and $K_2$-invariant. At last, consider the map
	\[
		\pi^f_{*,p}:= (f_* {H_1} \cap f_* {H_2})_{p}=f_{*,p}(H_1\cap H_2)_{f^{-1}(p)}\to T_{\pi(p)} M.
	\]
	It is clear that $\pi^f_{*,p}=\pi_{*,f^{-1}(p)}\circ f^{-1}_{*,p}$, where
	\[
		\pi_{*,f^{-1}(p)}: (H_1\cap H_2)_{f^{-1}(p)}\to T_{\pi(f^{-1}(p))} M=T_{\pi(p)} M
	\]
	is an isomorphism. So $\pi^f_{*,p}$ is an isomorphism. Thus the pair $(f_* H_1,f_* H_2)$ is a connection in $\mathbb{P}$. 
\end{proof}

\vspace*{2cm}

\end{document}